\documentclass[a4paper, 10pt]{amsart}
\usepackage{amsmath,amsthm}
\usepackage{amssymb, color, amscd, hyperref}
\usepackage[mathscr]{eucal}

\def\doi#1{{\small\href{https://doi.org/#1}{\path{doi:#1}}}}
\def\arxiv#1{{\small\href{http://www.arxiv.org/abs/#1}{\path{arXiv:#1}}}}
\def\url#1{{\small\href{#1}{\path{#1}}}}

\theoremstyle{plain}
\newtheorem{theorem}{\bf Theorem}[section]
\newtheorem{proposition}[theorem]{\bf Proposition}
\newtheorem{lemma}[theorem]{\bf Lemma}
\newtheorem{corollary}[theorem]{\bf Corollary}

\theoremstyle{definition}

\newtheorem{definition}[theorem]{\bf Definition}
\newtheorem{remark}[theorem]{\bf Remark}
\newtheorem{problem}[theorem]{\bf Problem}

\newcommand{\N}{\mathbb N}
\newcommand{\Z}{\mathbb Z}

\newcommand{\Q}{\mathbb Q}
\newcommand{\F}{\mathbb F}

 \DeclareMathOperator{\ord}{ord}

\DeclareMathOperator{\spec}{spec} \DeclareMathOperator{\supp}{supp}
 
\DeclareMathOperator{\Ker}{Ker} 
 \DeclareMathOperator{\End}{End}
\DeclareMathOperator{\Aut}{Aut} \DeclareMathOperator{\Gal}{Gal}

\newcommand{\DP}{\negthinspace : \negthinspace}
\newcommand{\red}{{\text{\rm red}}}
\renewcommand{\t}{\, | \,}

\newcommand{\mon}{\text{\rm mon}}

\newcommand{\id}{\text{\rm id}}
\newcommand{\LK}{\,[\![}
\newcommand{\RK}{]\!]}

\newcommand{\BF}{\text{\rm BF}}
\newcommand{\FF}{\text{\rm FF}}
\newcommand{\C}{\text{\rm C}}

\numberwithin{equation}{section}

\newcommand{\mf}{\mathfrak}
\newcommand{\rest}{\negthinspace \restriction\negthinspace}
\newcommand{\bs}{\boldsymbol}
\newcommand{\ms}{\mathscr}
\newcommand{\mc}{\mathcal}

\newcommand{\wh}{\widehat}
\newcommand{\mmod}{\negthickspace \mod}
\renewcommand{\P}{\mathbb P}

\subjclass[2010]{20M13; 11B50, 11B75, 11E12, 11R27}
\thanks{This work was supported by the Austrian Science Fund FWF (Project P33499-N) and by the National Natural Science Foundation of China (Grant No. 12001331)}

\begin{document}

\title[Weighted zero-sum sequences, norm monoids, and binary quadratic forms]{On monoids of  weighted zero-sum sequences and \\ applications to norm monoids in Galois number fields and binary quadratic forms}

\author{Alfred Geroldinger and Franz Halter-Koch and Qinghai Zhong}

\address{University of Graz, NAWI Graz \\
Institute for Mathematics and Scientific Computing \\
Heinrichstra{\ss}e 36\\
8010 Graz, Austria (A.~Geroldinger, F.~Halter-Koch, and Q.~Zhong)}

\address{School of Mathematics and Statistics, Shandong University of Technology, Zibo, Shandong 255000, China (Q.~Zhong)}

\email{alfred.geroldinger@uni-graz.at, franz.halterkoch@gmx.at,  qinghai.zhong@uni-graz.at}
\urladdr{https://imsc.uni-graz.at/geroldinger, https://imsc.uni-graz.at/zhong/}

\keywords{weighted zero-sum sequences, norm monoids, binary quadratic forms, sets of lengths, catenary degrees}

\begin{abstract}
Let $G$ be an additive finite abelian group and $\Gamma \subset \End (G)$ be a subset of the endomorphism group of $G$. A sequence $S = g_1 \cdot \ldots \cdot g_{\ell}$ over $G$ is a ($\Gamma$-)weighted zero-sum sequence if there are $\gamma_1, \ldots, \gamma_{\ell} \in \Gamma$ such that $\gamma_1 (g_1) + \ldots + \gamma_{\ell} (g_{\ell})=0$.  We construct transfer homomorphisms from norm monoids (of  Galois algebraic number fields with Galois group $\Gamma$) and from monoids of positive integers, represented by binary quadratic forms, to monoids of weighted zero-sum sequences. Then we study  algebraic and arithmetic properties of  monoids of weighted zero-sum sequences.
\end{abstract}

\maketitle


\medskip
\section{Introduction} \label{1}
\medskip

Weighted zero-sum problems over finite abelian groups were introduced by Adhikari et al. (see  \cite{A-C-F-K-P06, Ad-Ch08a, Ad-Ra06} for some first papers), and the idea to use homomorphisms as weights is due to Yuan and Zeng \cite{Ze-Yu11b}.  The focus of early work was on combinatorial aspects of the theory, on weighted analogs of the Davenport constant, the Erd{\H{o}}s-Ginzburg-Ziv constant, the Harborth constant, and others. For example, Gao's Theorem (stating that $\mathsf E (G) = |G| + \mathsf d (G)$) found a weighted analog, based on the DeVos-Goddyn-Mohar Theorem (see Chapter 16 in the monograph \cite{Gr13a}; for a sample of further papers with a combinatorial flavor, see \cite{Yu-Ze10b, Ad-Gr-Su12a, Gr-Ma-Or12a, Ma-Or-Ra-Sc13a, Go-Le-Ma13a, HK14b, Gr-He15a, Gr-Ph-Po13, Ol-Le-Go19a, Ad-He21a, A-H-M-S22, Br-Ri22a}).    Interactions of weighted zero-sum problems and coding theory are discussed in \cite{ Ma-Or-Sa-Sc15}.

The arithmetic of a Krull domain or, more generally, of a Krull monoid with class group $G$ is closely connected with the arithmetic of an associated monoid of zero-sum sequences over $G$, which again is a Krull monoid. Pushed forward by this connection, the arithmetic of monoids of zero-sum sequences received wide attention (see the monographs \cite{Ge-HK06a, Ge-Ru09} and the  survey \cite{Sc16a}). Arithmetic investigations of monoids of weighted zero-sum sequences over a finite abelian group were initiated by Schmid et al. \cite{B-M-O-S22}.

The present paper has three main objectives. In Section \ref{3}, we establish transfer homomorphisms from norm monoids (of Galois-invariant orders in algebraic number fields) and from monoids of positive integers, that can be represented by binary quadratic forms, to monoids of weighted zero-sum sequences (Theorems \ref{3.2}, \ref{3.5}, and Corollary \ref{3.6}). Roughly speaking, this means that the arithmetic of these monoids coincides with the arithmetic of the associated monoids of weighted zero-sum sequences (Corollary \ref{6.5}).

In Sections \ref{4} and \ref{5}, we study algebraic properties of monoids $\mathcal B_{\Gamma} (G)$ of weighted zero-sum sequences.
(Weakly) Krull monoids and (weakly) Krull domains as well as C-monoids and C-domains are well-studied objects in Multiplicative Ideal Theory and Factorization Theory. In Theorem \ref{4.4}, we show that $\mathcal B_{\Gamma} (G)$ is weakly Krull (resp. Krull or transfer Krull) only in exceptional cases. In Theorem \ref{5.1}\, we recall the known fact that $\mathcal B_{\Gamma} (G)$  is a C-monoid and provide it with a simple proof. Then we characterize seminormal monoids of $\Gamma$-weighted zero-sum sequences (Theorem \ref{5.2}), and as a consequence we determine their class semigroups (Theorem \ref{5.5}).

In Section \ref{6}, we study the arithmetic  of the monoid $\mathcal B_{\pm} (G)$ of plus-minus weighted zero-sum sequences. In particular, we compare it with the arithmetic of the monoid $\mathcal B (G)$ of zero-sum sequences over $G$ and we find similarities and striking differences. We end with a list of problems (Problems \ref{6.7} and \ref{6.10}).

\medskip
\section{Background on monoids} \label{2}
\medskip

We denote by $\mathbb P \subset \N \subset \N_0 \subset \Z$ the set of prime numbers, positive integers, non-negative integers, and integers. For $a, b \in \Z$, let $[a, b] = \{ x \in \Z \colon a \le x \le b\}$ be the discrete interval between $a$ and $b$. For a subset $A  \subset \Z$,   the set of distances $\Delta (A) \subset \N$ is the set of all $d \in \N$ for which there is $a \in A$ such that $A \cap [a, a+d] = \{a, a+d\}$.
Let $G$ be an additive abelian group, and let $A, B \subset G$ be subsets. Then we set $A^{\bullet} = A \setminus \{0\}$. We denote by $\langle A \rangle$  the subgroup generated by $A$ and by $A + B = \{a+b \colon a \in A , b \in B \}$ their sumset.

Suppose that $G$ is finite and let $r \in \N$. An $r$-tuple $(e_1, \ldots, e_r) \in G^r$ is said to be \,{\it independent}\, (briefly, the elements $e_1, \ldots, e_r$ are independent) if $e_1, \ldots, e_r$ are nonzero and
\[
\sum_{i=1}^r m_i e_i = 0 , \quad \text{with} \quad m_1, \ldots, m_r \in \Z, \quad \text{implies that} \quad m_1e_1 = \ldots = m_re_r = 0 \,.
\]
The tuple $(e_1, \ldots, e_r)$ is called a \,{\it basis}\, of $G$ if $(e_1, \ldots, e_r)$ is independent and $G = \langle e_1, \ldots, e_r \rangle$. For every $n \in \N$, we denote by $C_n$ a cyclic group of order $n$.

\smallskip

Throughout this paper, a \,{\it monoid}\, $H$ means a multiplicative commutative cancellative semigroup with identity element $1_H = 1 \in H$. We denote by $H^{\times}$ the group of invertible elements, by $\mathsf q (H)$ the quotient group of $H$ and by $H_{\red} = \{ a H^{\times} \colon a \in H \}$ the associated reduced monoid. We call $H$ reduced if $H^\times = \{1\}$.

\smallskip

We denote by
\begin{itemize}
\item $H' = \{ x \in \mathsf q (H) \colon \text{there is some $N \in \N$ such that} \ x^n \in H \ \text{for all} \ n \ge N \}$ the \,{\it seminormalization}, by

\item $\widetilde H = \{ x \in \mathsf q (H) \colon \text{there is some $n \in \N$ such that} \,x^n \in H \}$ the \,{\it root closure}, and by

\item $\widehat H = \{ x \in \mathsf q (H) \colon \text{there is some $c \in H$ such that} \ cx^n \in H \ \text{for all} \ n \in \N \}$ the \,{\it complete integral closure}\, of $H$.
\end{itemize}
Then $H \subset H' \subset \widetilde H \subset \widehat H \subset \mathsf q (H)$, and $H$ is said to be {\it seminormal, root closed, or completely integrally closed} if $H = H'$, $H = \widetilde H$, or $H = \widehat H$.

Concerning the ideal theory of monoids, our main reference is \,\cite{HK98} (note however that our monoids do not contain a zero element, and thus $\emptyset$ is an $s$-ideal of $H$). In particular, we denote by $s$-$\spec (H)$ the set of all prime $s$-ideals of $H$ and by $\mathfrak X (H)$ the set of all minimal nonempty prime $s$-ideals.

For a set $P$, let $\mathcal F (P)$ be the free abelian monoid with basis $P$. Then every element $a \in \mathcal F (P)$ can be uniquely written in the form
\[
a = \prod_{p \in P} p^{\mathsf v_p (a)} \quad \text{with} \quad \mathsf v_p (a) \in \N_0 \ \text{and} \ \mathsf v_p (a) = 0 \ \text{for almost all} \ p \in P \,.
\]
We call
\[
|a|= \sum_{p \in P} \mathsf v_p (a) \quad \text{the \,{\it length} \ and } \ \supp (a) = \{ p \in P \colon \mathsf v_p (a) > 0 \} \subset P \quad \text{the \,{\it support}\, of } \ a.
\]
If $H$ is a submonoid of a free abelian monoid $F$, then  $H' \subset \widetilde H \subset \widehat H \subset \widehat F = F$.

\medskip

\noindent
{\bf Arithmetic of monoids.} Let $H$ be a monoid. We denote by $\mathcal A (H)$ the set of its \,{\it atoms}\, (i.~e., irreducible elements), by $\mathsf Z (H) = \mathcal F ( \mathcal A (H_{\red})$ the \,{\it factorization monoid}\, and by
\[
\pi \colon \mathsf Z (H) \to H_{\red} \,, \quad \text{defined by} \quad \pi (u) = u \quad \text{for all} \ u \in \mathcal A (H_{\red}) \,,
\]
the \,{\it factorization homomorphism}\, of $H$. For an element $a \in H$, let
\[
\begin{aligned}
\mathsf Z_H (a) & = \mathsf Z (a) = \pi^{-1} (a) \subset \mathsf Z (H) \quad \text{be the \,{\it set of factorizations}\, and} \\
\mathsf L_H (a) & = \mathsf L (a) = \{ |z| \colon z \in \mathsf Z (a) \} \subset \N_0 \quad \text{the \,{\it set of lengths}\, of $a$} \,.
\end{aligned}
\]
We denote by $\mathcal L (H) = \{ \mathsf L (a) \colon a \in H \}$ the \,{\it system of sets of lengths}\, of $H$. Then $H$ is called
\begin{itemize}
\item {\it atomic}\, if $\mathsf Z (a) \ne \emptyset$ for all $a \in H$,

\item an \,{\it \FF-monoid}\, if $\mathsf Z (a)$ is finite and nonempty for all $a \in H$,

\item a \,{\it \BF-monoid}\, if $\mathsf L (a)$ is finite and nonempty for all $a \in H$,

\item {\it half-factorial}\, if $|\mathsf L (a)|=1$ for all $a \in H$,

\item a \,{\it Mori monoid}\, if it satisfies the ascending chain condition on divisorial ideals.
\end{itemize}
Finitely generated monoids (e.g,  monoids of $\Gamma$-weighted zero-sum sequences, as discussed below) are Mori monoids and Mori monoids are BF-monoids.
Let $z, z' \in \mathsf Z (H)$, say
\[
z = u_1 \cdot \ldots \cdot u_{\ell}v_1 \cdot \ldots \cdot v_m \quad \text{and} \quad z' = u_1 \cdot \ldots \cdot u_{\ell} w_1 \cdot \ldots \cdot w_n \,,
\]
where $\ell, m, n \in \N_0$, all $u_i,v_j, w_k \in \mathcal A (H_{\red})$, and $\gcd (z, z') = u_1 \cdot \ldots \cdot u_{\ell}$. Then we call $\mathsf d (z,z') = \max \{m,n\}$ the \,{\it distance}\, between $z$ and $z'$. Let $a \in H$ and $N \in \N_0 \cup \{\infty\}$. A finite sequence $z_0, \ldots, z_k \in \mathsf Z (a)$ is a called a \,{\it $($monotone$) \ N$-chain of factorizations}\, if $\mathsf d (z_{i-1}, z_i) \le N$ for all $i \in [1,k]$ (and either $|z_0| \le \ldots \le |z_k|$ or $|z_0| \ge \ldots \ge |z_k|$). We denote by $\mathsf c (a) \in \N_0 \cup \{\infty\}$ (by $\mathsf c_{\mon} (a) \in \N_0 \cup \{\infty\}$) the smallest $N \in \N_0 \cup \{\infty\}$ such that any two factorizations $z, z' \in \mathsf Z (a)$ can be concatenated by a (monotone) $N$-chain of factorizations from $\mathsf Z (a)$ and then we call
\[
\mathsf c (H)  = \sup \{\mathsf c (b) \colon b \in H \} \in \N_0 \cup \{\infty\}
\]
the \,{\it catenary degree}\, and $\mathsf c_{\mon} (H) = \sup \{\mathsf c_{\mon} (b) \colon b \in H \} \in \N_0 \cup \{\infty\}$ the \,{\it monotone catenary degree}\, of $H$.

\medskip

\noindent
{\bf $\Gamma$-weighted zero-sum sequences.} Let $G$ be an additive finite abelian group, $G_0 \subset G$ be a subset, and $\Gamma \subset \End (G)$ be a subset of the endomorphism group of $G$. An element $S = g_1 \cdot \ldots \cdot g_{\ell} \in \mathcal F (G_0)$, where $\ell \in \N_0$ and $g_1, \ldots, g_{\ell} \in G_0$, is called  a \,{\it sequence}\, over $G_0$. We call $|S| = \ell$ its \,{\it length}, \,$\sigma (S) = g_1 + \ldots + g_{\ell} \in G$ its \,{\it sum}, and
\[
\sigma_\Gamma(S) = \{\gamma_1(g_1) + \ldots + \gamma_\ell(g_\ell) \colon \gamma_1, \ldots, \gamma_\ell \in \Gamma\} \ \text{ the \,{\it set of $\Gamma$-weighted sums}\, of} \ S \,.
\]
For a sequence $T \in \mc F(G_0)$, we write $T \t S$ if $S=TU$ for some $U \in \mc F(G_0)$. Moreover, we define
\[
\Sigma (S) = \{ \sigma (T) \colon 1 \ne T \in \mathcal F (G), \ T \t S \}\quad \text{and} \quad \Sigma_{\Gamma} (S) = \bigcup_{1 \ne T \t S } \sigma_{\Gamma} (T).
\]
For a sequence $S = g_1 \cdot \ldots \cdot g_\ell\in \mc F(G_0)$ and $\gamma \in \Gamma$, we set $\gamma S = \gamma(g_1) \cdot \ldots, \cdot \gamma(g_\ell)$.\newline
A sequence $S \in \mc F(G_0 )$ is called
\begin{itemize}
\item a\, {\it zero-sum sequence}\, if $\sigma (S) = 0$,

\item {\it zero-sum free}\, if $0 \notin \Sigma (S)$,

\item a\, {\it $\Gamma$-weighted zero-sum sequence}\, if $0 \in \sigma_{\Gamma} (S)$, and

\item {\it $\Gamma$-weighted zero-sum free}\, if $0 \notin \Sigma_{\Gamma} (S)$.
\end{itemize}

\smallskip

Let $\mathcal B (G_0)$ denote the \,{\it monoid of zero-sum sequences over \,$G_0$}, and let $\mathcal B_{\Gamma} (G_0)$ denote the \,{\it monoid of $\Gamma$-weighted zero-sum sequences over $G_0$}. By  definition,
\[
\mc B_\Gamma(G_0) = \Bigl\{ g_1 \cdot \ldots \cdot g_\ell \in \mc F(G_0) \colon \sum_{i=1}^{\ell} \gamma_i(g_i)=0 \ \text{ for some } \ \gamma_1, \ldots, \gamma_{\ell} \in \Gamma \Bigr\},
\]
and a sequence $S \in \mc F(G_0)$ belongs to $\mc B_\Gamma(G_0)$ if and only if there is a factorization
\[
S = \prod_{\gamma \in \Gamma} S_\gamma, \ \text{ where } \ S_\gamma \in \mc F(G_0) \ \text{ for all } \ \gamma \in \Gamma, \quad \text{and} \quad\sum_{\gamma \in \Gamma} \gamma \big(\sigma(S_\gamma) \big)=0 \,.\tag{$\dagger$}
\]
By definition, $\mc B_\Gamma(G_0) = \{1\}$ if either $G_0= \emptyset$ or $\Gamma = \emptyset$, \ $\mathcal B (G_0) = \mathcal B_{\{\id_G\}} (G_0) \subset \mc F(G_0)$ and, if \,$\id_G  \in \Gamma$, then $\mathcal B (G_0) \subset \mathcal B_{\Gamma} (G_0) \subset \mathcal F (G_0)$ are submonoids. Their sets of atoms,  $\mathcal A \big( \mathcal B (G_0) \big)$ and $\mathcal A \big( \mathcal B_{\Gamma} (G_0) \big)$,  are both finite, and therefore $\mathcal B (G_0)$ and $\mathcal B_{\Gamma} (G_0)$ are reduced finitely generated monoids. Since their quotient groups are torsion-free, they are affine monoids (in the sense of \cite{Br-Gu09a}), and  they are FF-monoids with finite catenary degrees by \cite[Proposition 1.5.5 and Theorem 3.1.4]{Ge-HK06a}. We denote by

\begin{itemize}
\item
$\mathsf D \big( \mathcal B (G_0) \big) = \max \bigl\{ |U| \,\colon U \in \mathcal A ( \mathcal B (G_0) )\bigr\}$  the \,{\it $($large$)$ Davenport constant}\, of $\mc B(G_0)$, by

\smallskip

\item
$\mathsf D  \big( \mathcal B_{\Gamma} (G_0) \big) = \max \bigl\{ |U| \colon  U \in \mathcal A ( \mathcal B_{\Gamma}\, (G_0) ) \bigr\}$  the \,{\it $($large$)$ Davenport constant} of $\mathcal B_{\Gamma} (G_0)$, by

\smallskip

\item
$\mathsf d \big( \mathcal B (G_0) \big) = \max \bigl\{ |S| \colon S \in \mathcal F (G_0) \ \text{is zero-sum free} \bigr\}$\, the  \,{\it $($small$)$ Davenport constant}\, of $\mathcal B (G_0)$, and by

\smallskip

\item
$\mathsf d \big( \mathcal B_{\Gamma} (G_0) \big) = \max \bigl\{ |S| \colon S \in \mathcal F (G_0) \ \text{is $\Gamma$-weighted zero-sum free} \bigr\}$\, the \,{\it $($small$)$ Davenport constant}\, of $\mathcal B_{\Gamma} (G_0)$.
\end{itemize}
These definitions are consistent with \cite{B-M-O-S22}, and the idea to define Davenport constants of \BF-monoids stems from \cite[Section 2.5]{Cz-Do-Ge16a}. By \cite[Theorem 3.3]{B-M-O-S22}, we have
\begin{equation} \label{inequality-0}
1 + \mathsf d \big( \mathcal B_{\Gamma} (G) \big) \le \mathsf D \big( \mathcal B_{\Gamma} (G) \big)  \le \mathsf D \big( \mathcal B (G) \big) \,,
\end{equation}
and both inequalities can be strict. By definition,  $1 + \mathsf d \big( \mathcal B_{\Gamma} (G_0) \big)$ is the smallest integer $\ell \in \N$ such that every sequence $S \in \mathcal F (G_0)$ of length $|S|\ge \ell$ has a nontrivial $\Gamma$-weighted zero-sum subsequence. In the special case, when $\Gamma = \{\id_G\}$,  we use the abbreviations
\[
\mathsf D (G_0) = \mathsf D \big( \mathcal B (G_0) \big), \quad \text{and} \quad  \mathsf d (G_0) = \mathsf d \big( \mathcal B (G_0) \big) \,,
\]
and these are the usual Davenport constants.
If $S \in \mathcal F (G)$ is zero-sum free, then $(- \sigma (S))S \in \mathcal A (G)$, and $1 + \mathsf d (G) = \mathsf D (G)$.

Suppose that $G \cong C_{n_1} \oplus \ldots \oplus C_{n_r}$, with $r \in \N_0$ and $1 < n_1 \t \ldots \t n_r$. Then
\begin{equation} \label{daven-1}
\mathsf D^* (G) = 1 + \sum_{i=1}^r (n_i-1) \le \mathsf D (G) \le |G| \,,
\end{equation}
and the left inequality is an equality for $p$-groups, for $r \le 2 $, and others (\cite[Chapter 5]{Ge-HK06a}).
In case $\Gamma = \{\id_G, -\id_G \}$ we set $\sigma_{\pm} (S) = \sigma_{\Gamma} (S)$, \ $\Sigma_{\pm} (S) = \Sigma_{\Gamma} (S)$, \ $\mathcal B_{\pm} (G_0) = \mathcal B_\Gamma (G_0)$, \ $\mathsf D_{\pm} (G_0) = \mathsf D \big( \mathcal B_{\pm} (G_0) \big)$, \ $\mathsf d_{\pm} (G_0) = \mathsf d \big( \mathcal B_{\pm} (G_0) \big) $, and we speak of \,{\it plus-minus weighted zero-sum sequences}, of \,{\it plus-minus weighted zero-sum free sequences}, and of the \,{\it monoid of plus-minus weighted zero-sum sequences}. By \thetag{$\dagger$}, a sequence $S \in \mc F(G_0)$ belongs to $\mc B_\pm(G_0)$ if and only if $S = S'S''$ for some $S',\,S'' \in \mc F(G_0)$ such that $\sigma(S') =\sigma(S'')$, and then $\sigma(S) =\sigma(S')+\sigma(S'') \in 2G$.

If $t \in [0,r]$ is maximal such that $2 \nmid n_i$, then, by \cite[Corollary 6.8]{B-M-O-S22}, we have
\begin{equation} \label{daven-2}
1 + \sum_{i=1}^t (n_i-1) + \sum_{i=t+1}^r \frac{n_i}{2} \le \mathsf D_{\pm} (G) \le \mathsf D (G) \,.
\end{equation}
\medskip

\medskip
\noindent
{\bf Transfer homomorphisms.} Transfer homomorphisms (see Definition \ref{2.1} below) are a key tool in Factorization Theory. A transfer homomorphism $\theta \colon H \to B$ allows to pull back arithmetic properties from the (in general, simpler) monoid $B$ to the monoid $H$ (the initial object of interest). The classic example of a transfer homomorphism is the homomorphism from a Krull monoid to its associated monoid of zero-sum sequences. To be more precise, let $H$ be a Krull monoid with class group $G$ and let $G_0 \subset G$ denote the set of classes containing prime divisors. Then there is a transfer homomorphism $\theta \colon H \to \mathcal B (G_0)$, and thus all arithmetic investigations (mainly done with methods from additive combinatorics) of monoids of zero-sum sequences find their application in the arithmetic of Krull monoids.

\smallskip
\begin{definition} \label{2.1}
A monoid homomorphism  $\theta \colon H \to B$ is called a {\it transfer homomorphism} if it has the following properties.

\begin{enumerate}
\item[{\bf (T1)}] $B = \theta(H) B^\times$  and  $\theta^{-1} (B^\times) = H^\times$.

\item[{\bf (T2)}] If $u \in H$, \ $b,\,c \in B$  and  $\theta (u) = bc$, then there exist \ $v,\,w \in H$ \ such that \ $u = vw$,  $\theta (v) \in bB^{\times}$, and  $\theta (w) \in c B^{\times}$.
\end{enumerate}
\end{definition}

To define the catenary degree in the fibers of a transfer homomorphism, let $\theta \colon H \to B$ be a transfer homomorphism of atomic monoids, and let $\overline \theta \colon \mathsf Z (H) \to \mathsf Z (B)$ be the unique monoid homomorphism satisfying $\overline \theta ( u H^{\times}) = \theta (u) B^{\times}$ for every $u \in \mathcal A (H)$. For $a \in H$, let $\mathsf c (a, \theta)$ be the smallest $N \in \N_0 \cup \{\infty\}$ with the following property:

\begin{enumerate}
\item[]
If $z,\, z' \in \mathsf Z(a)$ and $\overline \theta (z) =
\overline \theta (z')$, then there exist factorizations $z=z_0, \ldots, z_k \in \mathsf Z(a)$ for some $k \in \N_0$ such that $\overline \theta (z_i) = \overline \theta (z)$ and \ $\mathsf d (z_{i-1}, z_i) \le N$ for all $i \in [1,k]$ \ (we say that $z$ and $z'$ can be concatenated by an $N$-chain in the fiber
\ $\mathsf Z_H (a) \cap \overline \theta ^{-1} (\overline \theta
(z)$)\,).
\end{enumerate}
We call
\[
\mathsf c (H, \theta) = \sup \{\mathsf c (a, \theta) \colon a \in H \} \in \N_0 \cup \{\infty\}
\]
the \,{\it catenary degree of \,$\theta$ in the fibres}.

\smallskip

\begin{proposition} \label{2.2}
Let $\theta \colon H \to B$ be a transfer homomorphism of atomic monoids and $a \in H$. Then
\begin{itemize}
\item
$\mathsf L_H (a) = \mathsf L_B ( \theta (a))$. In particular $\mathcal L (H) = \mathcal L (B)$.

\smallskip

\item
$\mathsf c ( \theta (a) ) \le \mathsf c (a) \le \max \{ \mathsf c ( \theta (a) ), \mathsf c (a, \theta) \}$ and
      $\mathsf c (B) \le \mathsf c (H) \le \max \{ \mathsf c (B), \mathsf c (H, \theta) \}$.

\smallskip

\item
If\, $\mathsf c (H, \theta) \le \mathsf c (B)$, then $\mathsf c (H)=\mathsf c (B)$ and $\mathsf c_{\mon} (H)=\mathsf c_{\mon} (B)$.
\end{itemize}
\end{proposition}

\begin{proof}
See \cite[Proposition 3.2.3, Theorem 3.2.5, and Lemma 3.2.6]{Ge-HK06a}.
\end{proof}

\smallskip

\begin{proposition} \label{2.3}
 Let\, $[\,\cdot\,] \colon P \to G$ be a surjective map between nonempty sets and $\Theta \colon \mc F(P) \to \mc F(G)$ the unique monoid homomorphism satisfying $\Theta(p) = [p]$ for all $p \in P$. Let $B \subset \mc F(G)$ be a submonoid, $H = \Theta^{-1}(B) \subset \mc F(P)$ and $\theta = \Theta \rest H\colon H \to B$. Then $\theta$ is a transfer homomorphism, and $\mathsf c(H,\theta) \le 2$.
\end{proposition}

\begin{proof}
We apply \cite[Theorem 3.2.8]{Ge-HK06a} with $D = \mc F(P)$, \ $T=U = \{ 1 \}$, and $P_0= \emptyset$. For $u,\,v \in D$,  we define $u \sim v$ if  $\Theta(u)=\Theta(v$). If $u,\,v \in D$, \ $u \sim v$ and $u\in H$, then $\Theta(u)=\Theta(v) \in B$ and thus $v \in H$. For $p,\,p' \in P$, we have $p \sim p'$ if and only if $[p] = [p'] \in G$, whence we may identify $\widetilde P=G$. Thus, it follows that $\widetilde D = \mc F(G)$,  $\widetilde{\bs \beta} = \Theta$, \ $\widetilde H = B$,  and $\bs \beta = \theta$.
\end{proof}

\medskip
\section{Galois-invariant orders in number fields and binary quadratic forms} \label{3}
\medskip

The Davenport constant $\mathsf D (G)$ allows the following arithmetic interpretation which goes back to the 1960s (\cite{Ro63}). If $K$ is an algebraic number field, $\mathcal O_K$ its ring of integers, and $G$ its ideal class group, then $\mathsf D (G)$ is the largest $\ell \in \N$ for which there is an irreducible element $a \in \mathcal O_K$ such that $a \mathcal O_K = \mathfrak p_1 \cdot \ldots \cdot \mathfrak p_{\ell}$, where $\mathfrak p_1, \ldots, \mathfrak p_{\ell}$ are prime ideals of $\mathcal O_K$. B.~Schmid found a  link between Invariant Theory and Zero-Sum Theory by observing that the Noether number (from Invariant Theory) equals the Davenport constant (see \cite{Sc91a} and \cite[Section 4.3]{Cz-Do-Ge16a}). Halter-Koch gave an arithmetic interpretation of the (small) weighted Davenport constant $\mathsf d_{\Gamma} (G)$
(\cite{HK14b}),  and  W.A.~Schmid et al. established a transfer homomorphism from monoids, defined by norms of integers of a Galois algebraic number field, to monoids of weighted zero-sum sequences  (\cite{B-M-O-S22}). In this section, we go far beyond these former results. We consider monoids, defined by norms of (totally positive) elements of Galois-invariant orders in a Galois algebraic number field, and construct transfer homomorphisms onto monoids of weighted zero-sum sequences (Theorem \ref{3.2}). Recall that norm groups play a crucial role in  class field theory (\cite{HK22a}). Our approach, via not necessarily maximal orders of number fields, is the basis for the  study of monoids of positive integers, represented by integral binary quadratic forms (Theorem \ref{3.5} and Corollary \ref{3.6}).

\smallskip
Let $K$ be a Galois algebraic number field, $\Gamma = \Gal(K/\Q)$ its Galois group,
$\mathcal O_K$ its ring of integers, $\mc O$ a Galois-invariant order in $K$, and $\mf m = (\mc O \DP \mathcal O_K)$ the conductor of $\mc O$. Let $K^+$ be the group of  totally positive elements of $K$ and $\mc O^+= \mc O \cap K^+$. We consider simultaneously the totally positive and the general situation, and we write $(\ast)^{(+)}$ if we mean both $(\ast)$ or $(\ast)^+$.

We denote by $\ms I(\mc O)$ the group of invertible fractional ideals of $\mc O$,  by $\ms P (\mc O)$ the set of all nonzero prime ideals of $\mc O$, by $\ms H^{(+)}(\mc O)$ the subgroup of fractional principal ideals $a \mc O$ with $a \in K^{(+)}$, and by $\mc C^{(+)}(\mc O) = {\rm Pic}^{(+)}(\mc O) = \ms I(\mc O)/\ms H^{(+)}(\mc O)$ the (narrow) class group of $\mc O$. The groups $\ms I(\mc O)$, \ $\ms H^{(+)}(\mc O)$, and $\mc C^{(+)}(\mc O)$ are $\Gamma$-modules, and we write $\mc C^{(+)}(\mc O)$ additively \,(then $\bs 0 = [\mc O]^{(+)} = \ms H^{(+)}(\mc O)\in \mc C^{(+)}(\mc O)$). For an invertible fractional ideal $\mf a$ of $\mc O$, we denote by $[\mf a]^{(+)} \in \mc C^{(+)}(\mc O)$ the (narrow) ideal class containing $\mf a$.

For a $\Gamma$-invariant ideal $\mf n$ of $\mc O_K$  such that $\mf n \subset \mf m$, we denote by $\ms P_\mf n(\mc O)$ the set of all prime ideals of $\mc O$ coprime to $\mf n$ and by $\ms I_\mf n(\mc O)$ the monoid of all ideals of $\mc O$ coprime to $\mf n$.
For a positiv integer $m \in \N \cap \mf m$, we denote by $\P_m$ the set of all primes not dividing $m$ and by $\N_m$ the monoid of all positive integers coprime to $m$. Apparently, $\N_m = \mc F(\P_m)$ is the free abelian monoid with basis $\P_m$. We set $\ms P_m(\mc O) = \ms P_{m\mc O_K}(\mc O)$, \ $\ms I_m(\mc O) = \ms I_{m\mc O_K}(\mc O)$, and we denote by $\ms H_m^{(+)}(\mc O) = \ms H^{(+)}(\mc O) \cap \ms I_m(\mc O)$ the monoid of all principal ideals $a\mc O$ coprime to $m\mc O_K$, where $a \in \mc O^{(+)}$.

\smallskip
\begin{lemma} \label{3.1}
Let $m \in  \N \cap \mf m$.
\begin{enumerate}
\item
The maps
\[
\nu \colon \begin {cases}
\ms I_m(\mc O_K) &\to \ \ms I_m(\mc O) \\
\quad\mf A &\mapsto \ \mf A \cap \mc O
\end{cases} \qquad \text{and} \qquad
j \colon \begin {cases} \ms I_m(\mc O) &\to \  \ms I_m(\mc O_K) \\ \quad\mf a &\mapsto \ \mf a \mc O_K
\end{cases}
\]
are mutually inverse $\Gamma$-isomorphisms such that $j(\ms P_m(\mc O)) = \ms P_m(\mc O_K)$, and $\ms I_m(\mc O)$ is the free abelian monoid with basis $\ms P_m(\mc O)$. Moreover, $\mc O/\mf a \cong \mc O_K/\mf a \mc O_K$ for every ideal $\mf a \in \ms I_m(\mc O)$.

\smallskip

\item
$\ms I_m(\mc O) \subset \ms I(\mc O)$, and if \,$\mf p \in \ms P_m(\mc O)$, then $\mc O_\mf p = (\mc O_K)_\mf p = (\mc O_K)_{\mf p\mc O_K}$ is a discrete valuation domain $($we denote by
$\mathsf v_\mf p = \mathsf v_{\mf p\mc O_K}$ the associated discrete  valuation of\, $K$, and we call $\dim_{\F_p}(\mc O_K/\mf p\mc O_K)$ the \,{\it degree}\, of \,$\mf p)$. For an ideal $\mf a \in \ms I_m(\mc O)$, we set $\mathsf v_\mf p (\mf a) = \mathsf v_\mf p(\mf a \mc O_K)$, and then
\[
\mf a = \prod_{\mf p \in\ms P_m(\mc O)} \mf p^{\mathsf v_\mf p(\mf a)}.
\]

\smallskip

\item
The group \,$\mc C^{(+)}(\mc O)$ is finite, and every class $g \in \mc C^{(+)}(\mc O)$ contains infinitely many prime ideals of degree $1$ of $\mc O$.
\end{enumerate}
\end{lemma}

\begin{proof}
1. Obviously, $j$ is a monoid homomorphism satisfying $\gamma (\mf a) \mc O_K = \gamma (\mf a \mc O_K)$  for all $\gamma \in \Gamma$. By \cite[Theorem 2.11.3]{HK20a}, we have $(\mf A\cap \mc O)\mc O_K = \mf A$ for all $\mf A \in \ms I_\mf m(\mc O_K)$, and $\mf a \mc O_K \cap \mc O= \mf a$ for all $\mf a \in \ms I_\mf m(\mc O)$. As $\ms I_m(\mc O) \subset \ms I_\mf m(\mc O)$, it suffices to prove that $\mf A \in \ms I_m(\mc O_K)$ implies $\mf A \cap \mc O \in \ms I(\mc O)$, and that $\mf a \in \ms I_m(\mc O)$ implies $\mf a \mc O_K\in \ms I_m(\mc O_K)$ \,(for then $j$ is bijective and $\nu$ is its inverse).

If $\mf A \in \ms I_m(\mc O_K)$, then $\mf A + m\mc O_K=\mc O_K$, and as $m \mc O_K\subset \mc O$, the modular law implies that $\mf A \cap \mc O +m\mc O_K = (\mf A+m\mc O_K) \cap \mc O = \mc O$.

If $\mf a \in \ms I_m(\mc O)$, then $\mf a + m\mc O_K = \mc O$, hence $1 \in \mf a +m\mc O_K \subset \mf a \mc O_K+m\mc O_K$, and $\mf a\mc O_K+m\mc O_K= \mc O_K$.

Since $\ms P_m(\mc O) = \ms P_\mf m(\mc O) \cap \ms I_m(\mc O)$ and $j(\mc P_\mf m
(\mc O)) = \ms P_\mf m(\mc O_K)$ by \cite[Theorem 2.11.4]{HK20a}, it follows that $j(\ms P_m(\mc O)) = \ms P_m(\mc O_K)$. In particular, $\nu(\ms P_m(\mc O_K)) = \ms P_m(\mc O)$, and as $\ms I_m(\mc O_K)$ is the free abelian monoid with basis $\ms P_m(\mc O_K)$, it follows that $\ms I_m(\mc O)$ is the free abelian monoid with basis $\ms P_m(\mc O)$.

The isomorphism $\mc O/\mf a \cong \mc O_K/\mf a \mc O_K$ for every ideal $\mf a \in \ms I_m(\mc O)$ holds by \cite[Theorem 2.11.3]{HK20a},

\smallskip

2. We obtain $\ms I_m(\mc O)\subset \ms I_{\mf m}(\mc O) \subset \ms I(\mc O)$ by \cite[Theorem 2.11.3]{HK20a}. If $\mf p \in \ms P_m(\mc O)$, then $\mf m \not \subset \mf p$, and therefore $\mc O_\mf p = (\mc  O_K)_\mf p$ by \cite[Theorem 2.11.2]{HK20a}. Since $\mc O \setminus \mf p \subset \mc O_K \setminus \mf p \mc O_K$, we get $(\mc O_K)_\mf p \subset(\mc O_K)_{\mf p\mc O_K}$. As to the reverse inclusion, let $x = s^{-1}a\in (\mc O_K)_{\mf p \mc O_K}$, where $a \in \mc O_K$ and $s \in \mc O_K \setminus \mf p\mc O_K$. If $t \in \mf m \setminus \mf p$, then $st \in \mc O \setminus \mf p$, and $x = (st)^{-1}(at) \in \mc O_\mf p$. To prove the product formula for $\mf a \in \ms I(\mc O)$, we calculate
\[
\Bigl[\prod_{\mf p \in \ms P_m(\mc O)}\mf p^{\mathsf v_\mf p(\mf a)}\Bigr]\mc O_K =\negthinspace \prod_{\mf p \in \ms P_m(\mc O)}(\mf p\mc O_K)^{\mathsf v_{\mf p\mc O_K}(\mf a\mc O_K)}= \negthinspace\prod_{\mf P \in \ms P_m(\mc O_K)}\mf P^{\mathsf v_\mf P(\mf a \mc O_K)}= \negthinspace\prod_{\mf P \in \ms P(\mc O_K)}\mf P^{\mathsf v_\mf P(\mf a \mc O_K)} = \mf a \mc O_K,
\]
and as $j$ is an isomorphism, it follows that
\[
\mf a = \prod_{\mf p \in\ms P_m(\mc O)} \mf p^{\mathsf v_\mf p(\mf a)}.
\]

\smallskip

3. Let $\mc C_K^\mf m$ be the ray class group and $\mc C_K^{\circ \mf m}$ the small ray class group modulo $\mf m$ of $K$ as defined in \cite[Definition 3.6.6]{HK20a}. By \cite[Theorem 3.6.7]{HK20a}, these groups are finite, and for an ideal $\mf A \in \ms I_\mf m(\ms O_K)$ we denote by $[\mf A]_\mf m \in \mc C_K^\mf m$ the ray class and by $[\mf A]_\mf m^\circ \in \mc C^{\circ \mf m}_K$ the small ray class containing $\mf A$. By \cite[Theorem 2.11.6]{HK20a}, there exist epimorphisms $\rho^+ \colon \mc C_K^\mf m \to \mc C^+(\mc O)$ and $\rho \colon \mc C_K^{\circ\mf m} \to \mc C(\mc O)$ such that $\rho^+([\mf A]_\mf m) = [\mf A \cap \mc O]^+$ and $\rho([\mf A]_\mf m^\circ) = [\mf A \cap \mc O]$ for every ideal $\mf A \in \ms I_\mf m(\mc O_K)$. Hence the groups $\mc C(\mc O)$ and $\mc C^+(\mc O)$ are finite.

Since every class $g \in \mc C(\mc O)$ is composed of classes in $\mc C^+(\mc O)$, it suffices to prove that every class in $\mc C^+(\mc O)$ contains infinitely many prime ideals of degree $1$. Thus, let $g \in \mc C^+(\mc O)$ and $C \in \mc C_K^\mf m$ such that $\rho^+(C)=g$. By \cite[Theorem 4.4.2]{HK20a}, the set $\ms P_m(\mc O_K)\cap C$ has positive Dirichlet density, and therefore there exist infinitely many prime ideals $\mf P$
of degree $1$ in $\ms P_\mf m(\mc O_K) \cap C$, and for these the prime ideals $\mf P  \cap \mc O \in \ms P_\mf m(\mc O)$ are prime ideals of degree $1$ in $g$.
\end{proof}

\medskip

Let again $m \in \N \cap \mf m$. If $\mf p \in \ms P(\mc O)$, then $\mf p \cap \Z = \mf p\mc O_K \cap \Z = p\Z$, where $p$ is the (unique) prime lying in $\mf p$. Then $p \in \P_m$ holds if and only if $\mf p \in \ms P_m(\mc O)$, and we say that $\mf p$ lies above $p$.

For an ideal $\mf a \in \ms I_m(\mc O)$, we set $\mf N(\mf a)= \mf N(\mf a \mc O_K) = |\mc O/\mf a|\in \N_m$. The map $\mf N \colon \ms I_m(\mc O) \to \N_m$ is a monoid homomorphism, and if $\mf p \in \ms P(\mc O)$ and $\mf p \cap \Z = p\Z$, then $\mf N(\mf p) = p^{f_p}$ , where $f_p \in \N$ is the inertia degree of $p$ in $K$. Hence
\[
\mf N(\ms I_m(\mc O)) =\mf N(\ms I_m(\mc O_K))=\{n \in \N \colon f_p\t \mathsf v_p(n) \ \text{ for all } \ p \in \P_m\}.
\]
is the free abelian monoid with basis $\{p^{f_p} \colon p \in\P_m\}$ which only depends on $K$, and we set
\[
N_m(K) = \mf N(\ms I_m(\mc O_K)) \quad \text{and} \quad
N_m^{(+)}(\mc O)=\mf N(\ms H_m^{(+)}(\mc O)) = \{ |\mathsf N_{K/\Q}(a)| \colon 0 \ne a \in \mc O^{(+)}\} \,.
\]

\medskip
Now we can start with the construction of  the announced transfer homomorphism from the monoid $N_m^{(+)}(\mc O) \subset N_m(K)$  onto the monoid $\mc B_\Gamma \big(\mc C^{(+)}(\mc O) \big)$ of $\Gamma$-weighted  zero-sum sequences.

For every prime $p \in \P_m$, we fix a prime ideal $\mf p_p \in \ms P_m(\mc O)$ above $p$, and since $\mc C^{(+)}(\mc O)$ is finite and every class contains infinitely many prime ideals of degree $1$, we can do this in such a way that $\{[\mf p_p]^{(+)} \colon p \in \P_m, \ f_p=1\} = \mc C^{(+)}(\mc O)$. If $p \in \mf p_p$, then $\{\mf P \in \ms P_m(\mc O_K) \colon p \in \mf P\} = \{\gamma \mf p_{p} \mc O_K \colon \gamma \in \Gamma\}$ by \cite[Theorem 2.5.4]{HK20a}, hence \,$\{\mf p \in \ms P_m(\mc O) \colon p \in \mf p\}= \{\gamma \mf p_p \colon \gamma \in \Gamma\}$, and $\mf N(\gamma\mf p_{p}) = p^{f_p}$ for all $\gamma \in \Gamma$.

We denote by $\mc F(\mc C^{(+)}(\mc O))$ the  (multiplicative) free abelian monoid with basis $\mc C^{(+)}(\mc O)$ and define
\[
\Theta \colon N_m(K) \to \mc F(\mc C^{(+)}(\mc O)) \quad \text{by} \quad \Theta (n) = \prod_{p\in \P_m} \bigl([\mf p_p]^{(+)}\bigr)^{\mathsf v_p(n) /f_p} \in \mc F(\mc C^{(+)}(\mc O)).
\]
Then $\Theta$ is a monoid homomorphism, it is surjective, since $\{[\mf p_p]^{(+)} \colon p \in \P_m, \ f_p=1 \} = \mc C^{(+)}(\mc O)$, and $\Theta^{-1}(\mathsf 1)=\{1\}$.

\smallskip
\begin{theorem} \label{3.2}~
\begin{enumerate}
\item
Let $n \in N_m(K)$. Then \ $n \in N_m^{(+)}(\mc O)\,\iff\,\Theta(n) \in \mc B_\Gamma(\mc C^{(+)}(\mc O))$.

\smallskip

\item
The map $\theta = \Theta\restriction N_m^{(+)}(\mc O)\colon N_m^{(+)}(\mc O)\,\to\, \mc B_\Gamma(\mc C^{(+)}(\mc O))$ is a transfer homomorphism  satisfying \,$\mathsf c(N_m^{(+)}(\mc O), \theta)\le 2$.
\end{enumerate}
\end{theorem}

\begin{proof}
We set $G = \mc C^{(+)}(\mc O)$, and for every class $g \in G$ we set $\P_g = \{ p \in \P_m \colon [\mf p_p]^{(+)} = g\}$. By our choice of the family $(\mf p_p)_{p \in \P_m}$ it follows that $\P_g \ne \emptyset$ for all $g \in G$.

\smallskip

1. \,$\Rightarrow$:\, Let $n = \mf N(a\mc O)\in N_m^{(+)}(\mc O)$, where $a \in \mc O^{(+)}$ and $a\mc O\in \ms I_m(\mc O)$, say
\[
a\mc O = \prod_{\mf p \in \ms P_m(\mc O)} \mf p^{\mathsf v_\mf p(a)} = \prod_{p \in \P_m} \prod_{\gamma \in \Gamma} \gamma \mf p_p^{\mathsf v_{\gamma \mf p_p}(a)}, \quad \text{and} \quad 0= \sum_{\gamma\in \Gamma} \gamma \sum_{p \in\P_m}\mathsf v_{\gamma \mf p_p}(a)[\mf p_p]^{(+)}  \in G.
\]
Then we obtain
\[
n = \prod_{p \in \P} \prod_{\gamma \in \Gamma}p^{f_p \mathsf v_{\gamma \mf p_p}(a)}, \quad \text{and therefore} \quad \frac{\mathsf v_p(n)}{f_p}= \sum_{\gamma \in \Gamma}\mathsf v_{\gamma \mf p_p}(a) \ \text{ for all } \ p \in \P_m.
\]
It follows that
\[
\Theta(n) = \prod_{\mf p \in \ms P_m(\mc O)} \bigl([\mf p_p]^{(+)}\bigr)^{\mathsf v_p(n)/f_p} = \prod_{\gamma \in \Gamma}\prod_{p \in \P_m} \bigl([\mf p_p]^+\bigr)^{\mathsf v_{\gamma \mf p_p}(a)}\,\in \mc F(G),
\]
and since
\[
\sum_{\gamma \in \Gamma} \gamma \,\sigma \Bigl( \prod_{p \in \P_m}\bigl([\mf p_p]^+\bigr)^{\mathsf v_{\gamma \mf p_p}(a)} \Bigr)= \sum_{\gamma\in \Gamma} \gamma \sum_{p \in\P_m} \mathsf v_{\gamma \mf p_p}(a)\,[\mf p_p]^{(+)} =0\in G,
\]
we eventually see that $\Theta(n) \in\mc B_{\Gamma}(G)$.

\smallskip

$\Leftarrow$:\, Let
\[
\Theta(n) = \prod_{p \in \P_m}\bigl([\mf p_p]^{(+)}\bigr)^{\mathsf v_p(n)/f_p} \in \mc B_\Gamma(G) \quad \text{and} \quad \Theta(n) = \prod_{\gamma \in \Gamma} S_\gamma, \quad \text{where} \quad \sum_{\gamma \in \Gamma} \gamma \sigma(S_\gamma)=0.
\]
We set
\[
\Theta(n) =  \prod_{g \in G} g^{N_g} \quad \text{and, for $\gamma \in \Gamma$,} \quad S_\gamma = \prod_{g \in G} g^{N_{\gamma, g}}, \quad \text{where} \quad N_g,\, N_{\gamma, g} \in \N_0.
\]
For every $g \in G$ it follows that
\[
N_g = \sum_{p \in \P_g} \frac{\mathsf v_p(n)}{f_p} = \sum_{\gamma \in \Gamma}N_{\gamma,g}.
\]
For $g \in G$ and $\gamma \in \Gamma$ we split $N_{\gamma, g}$ such that
\[
N_{\gamma, g} = \sum_{p \in \P_g}N_{\gamma, p}\quad \text{and} \quad \sum_{\gamma \in \Gamma} N_{\gamma,p} = \frac{\mathsf v_p(n)}{f_p} \ \text{ for all } \ p \in \P_g.
\]
Now we set
\[
\mf a = \prod_{g \in G} \prod_{p \in \P_g} \prod_{\gamma \in
\Gamma} (\gamma \mf p_p)^{N_{\gamma,p}}\, \in \ms I_m(\mc O), \quad \text{whence} \quad \mf N(\mf a) = \prod_{g \in G} \prod_{p \in \P_g}\prod_{\gamma \in \Gamma} p^{f_pN_{\gamma,  p}}.
\]
If $g \in G$ and $p \in \P_g$, then
\[
\mathsf v_p(\mf N(\mf a)) = \sum_{\gamma \in \Gamma}f_p N_{\gamma,p} = \mathsf v_p(n),
\]
and therefore $\mf N(\mf a) = n$.
Since
\[
[\mf a]^{(+)} = \sum_{g \in G} \sum_{p \in \P_g} \sum_{\gamma \in \Gamma} N_{\gamma,p} \gamma g= \sum_{g \in G}\sum_{\gamma \in \Gamma} N_{\gamma, g} \gamma g = \sum_{\gamma \in \Gamma} \gamma \sigma(S_\gamma) = 0,
\]
we obtain $\mf a \in \ms H_m^{(+)}(\mc O)$, and consequently $n \in N_m^{(+)}(\mc O)$.

\smallskip

2. Apply Proposition \ref{2.3} with $P = \{p^{f_p} \colon p \in\P_m\}$, \ $G = \mc C^{(+)}(\mc O)$, \ $B =\mc B_\Gamma(\mc C^{(+)}(\mc O)$, \ $\mc F(P) = N_m(K)$, and $H = \Theta^{-1}(B) = N_m^{(+)}(\mc O)$.
\end{proof}

\bigskip
\noindent
{\bf Representations by binary quadratic forms.} Binary quadratic forms are a classic topic in number theory that can be traced back to Gauss. It is closely connected with the ideal theory of quadratic orders and the theory of continued fractions. For an integer $m \ge 2$, let $R_m^\circ(\Delta)$ denote the monoid of all integers $n \in \N_m$ that are represented by the principal class of the composition class group of not negative definite primitive integral quadratic forms of discriminant $\Delta = \mathsf d_Km^2$. We show that this monoid allows a transfer homomorphism to a monoid of weighted zero-sum sequences (Theorem \ref{3.5}). To do so, we substantially use the relationship between the theory of quadratic forms and quadratic orders (as presented in \cite[Chapters 5 and 6]{HK13a}) and the results in the first part of this section.

\smallskip
Let $\Delta$ be a quadratic discriminant, i.e., $\Delta \in \Z$ is not a square and $\Delta \equiv 0 \text{ or } 1\mmod 4$. We set $\Delta = 4D+s$, where $D \in \Z$ and $s \in \{0,1\}$ and consider the order
\[
\mc O_\Delta = \Z\Bigl[ \frac{s+\sqrt \Delta}2\Bigr]\quad \text{in the quadratic number field } \ K =\Q(\sqrt \Delta).
\]
Let $\tau$ be the non-trivial automorphism of $K$ so that $\Gamma = \Gal(K/\Q) = \{ \id_K,\tau\}$.
If $\mathsf d_K$ denotes the discriminant of $K$, then $\Delta = \mathsf d_Km^2$, where $m\mc O_K$ is the conductor of $\mc O_\Delta$.

\bigskip
Let $\mf F_\Delta$ be the (additively written) composition class group of not negative definite primitive integral quadratic forms of discriminant $\Delta$. For a not negative definite primitive integral quadratic form $f = aX^2 +bXY+cY^2 \in \Z[X,Y]$ with discriminant $\Delta = b^2-4ac$, we denote by $\LK a,b,c\RK \in \mf F_\Delta$ its equivalence class. The principal class $O=\LK 1,s,-D \RK\in \mf F_\Delta$ is the zero of $\mf F_\Delta$, and if $F = \LK a,b,c \RK\in \mf F_\Delta$, then $-F = \LK a, -b,c\RK$.

We say that a binary quadratic form $f\in \Z[X,Y]$ [properly] represents an integer $n$ if there exist $x,\,y \in \Z$ such that $n = f(x,y)$ \,[and $\gcd (x,y)=1$]. We say that a class $F \in \mf F_\Delta$ [properly] represents $n$ if some (and then every) form $f \in F$ [properly] represents $n$. Obviously, $F$ represents $n$ if and only if $n = c^2n_0$ for some $c, \,n_0 \in \N$ such that $F$ properly represents $n_0$.

We denote by $R_m(\Delta)$ the set of all $n \in \N_m$ which are represented by a class $F \in \mf F_\Delta$ and by $R_m^\circ (\Delta)$ the set of all $n \in R_m(\Delta)$ which are represented by the principal class $O \in \mf F_\Delta$. The below-mentioned Theorem \ref{3.4} is the link between ideal theory of $\mc O_\Delta$ and the representation by binary quadratic forms. Before however (and also for later use) we need the subsequent preparations \,{\bf A}, \,{\bf B}\, and Lemma \ref{3.3}.

\medskip

{\bf A.} \,By \cite[Theorem 6.4.2]{HK13a},  there is an isomorphism $\Phi_\Delta \colon \mf F_\Delta \to \mc C^+(\mc O_\Delta)$ such that
\[
\Phi_\Delta(F) = \Bigl[\Bigl(a, \frac{b + \sqrt \Delta}2\Bigr)\Bigr]^+\in \mc C^+(\mc O_\Delta)\quad \text{for every class } \ F = \LK a,b,c\RK \in \mf F_\Delta \ \text{ with } \ a>0.
\]
In particular, $\Phi_\Delta(O)=\bs 0$, and if $F= \LK a,b,c \RK$, then
\[
\Phi_\Delta (-F) = \Bigl[\Bigl(a, \frac{-b+\sqrt \Delta}2\bigr)\Bigr]^+= \Bigl[\Bigl(a, \frac{b-\sqrt \Delta}2\bigr)\Bigr]^+ = \tau \Phi_\Delta(F).
\]

\smallskip

{\bf B.}\, An ideal $\mf a$ of $\mc O_\Delta$ is called \,{\it primitive}\, if $c^{-1} \mf a \not \subset \mc O_\Delta$ for every integer $c \ge 2$, and it is called \,{\it regular}\, if it is primitive and invertible. Consequently, an ideal $\mf a$ of $\mc O_\Delta$ is invertible if and only if $\mf a = e \mf a_0$ for some $e \in \N$ and some regular ideal $\mf a_0$. The following Lemma \,\ref{3.3}\, is a slight refinement of \cite[Theorem 6.4.13]{HK13a}.

\smallskip
\begin{lemma} \label{3.3}
Let $F \in \mf F_\Delta$ and $n \in \N$. Then $n$ is $($properly$)$  represented by $F$ if and only if $n = \mf N(\mf a)$ for some $($regular$)$ ideal $\mf a \in \Phi_\Delta(F)$.
\end{lemma}

\begin{proof} The assertion concerning proper representation is just \cite[Theorem 6.4.13]{HK13a}.

Let now $n$ be represented by the class $F$. Then $n = n_0d^2$, where $n_0,\,d \in \N$, and $n_0$ is properly represented by $F$. Hence $n_0 = \mf N(\mf a_0)$ for some regular ideal $\mf a_0 \in \Phi_\Delta(F)$. Then $n = \mf N(d\mf a_0)$, and $d\mf a_0 \in [d\mf a_0]^+ = [\mf a_0]^+ = \Phi_\Delta(F)$.

As to the converse, let $n = \mf N(\mf a)$ for some ideal $\mf a \in \Phi_\Delta(F)$. Then $\mf a$ is invertible, hence $\mf a = d \mf a_0$, where $d \in \N$ and $\mf a_0$ is a regular ideal. Hence $n_0 = \mf N(\mf a_0)$ is properly represented by $F = \Phi_\Delta^{-1}([\mf a_0]^+) = \Phi_\Delta^{-1}([\mf a]^+)$, \ $n = d^2n_0$ is represented by $F$ and $\mf a \in \Phi_\Delta(F)$.
\end{proof}

\medskip

In the sequel, we adopt all the notation introduced at the beginning of this section.

\smallskip
\begin{theorem} \label{3.4}
$R_m(\Delta) = N_m(K)$\, and\, $R_m^\circ(\Delta) = N_m^+(\mc O_\Delta)$. In particular, $R_m(\Delta)$ is the free abelian monoid with basis $\{p^{f_p} \colon p \in \P_m\}$, and $R_m^\circ(\Delta)$ is a submonoid.
\end{theorem}

\begin{proof}
If $n \in R_m(\Delta)$ is represented by the class $F \in \mf F_\Delta$, then \,Lemma \ref{3.3}\, shows that $n = \mf N(\mf a)$ for some ideal $\mf a \in \Phi_\Delta(F) \cap \ms I_m(\mc O_\Delta)$, and consequently $n \in N_m(K)$. If $n \in R_m^\circ(\Delta)$, then $F=O$, hence $\mf a\in \Phi_\Delta(O) = \bs 0$, and thus $\mf  a \in \ms H^+_m(\mc O_\Delta)$,   i.~e., $n \in N_m^+(\mc O_\Delta)$.

As to the reverse inclusion, let $n = \mf N(\mf a)\in N_m(K)$ for some $\mf a \in \ms I_m(\mc O_\Delta)$. By Lemma \ref{3.3}, $n$ is represented by $F = \Phi_\Delta^{-1}([\mf a]^+)$, i.~e., $n \in R_m(\Delta)$. If $n \in N_m^+(\mc O_\Delta)$, then $\mf a \in \ms H^+_m(\mc O_\Delta)$, hence $[\mf a]^+ = \bs 0$ and $F=O$, i.~e., $n \in R_m^\circ(\Delta)$.
\end{proof}

\smallskip

For every prime $p \in \P_m$, we fix a prime ideal $\mf p_p \in \ms P_m(\mc O_\Delta)$ lying above $p$ such that
\[
\mc C^+(\mc O_\Delta) = \{[\mf p_p]^+ \colon p \in \P_m, \ f_p=1\}.
\]
As usual, $\bigl(\frac \Delta p \bigr)$ denotes the Kronecker symbol. By \cite[Theorem 5.8.8]{HK13a} we obtain:

\smallskip

\begin{itemize}
\item
If $\bigl(\frac \Delta p \bigr) = 1$, then $f_p=1$, \ $\tau \mf p_p \ne \mf p_p$, \ $p\mc O_\Delta = \mf p_p\tau(\mf p_p)$, \  $\tau([\mf p_p]^+) = [\tau \mf p_p]^+ = -[\mf p_p]^+$ in $\mc C^+(\mc O_\Delta)$.

\smallskip

\item
If $\bigl(\frac \Delta p \bigr) = -1$, then $f_p=2$, \ $\mf p_p=p\mc O_\Delta$ and $[\mf  p_p]^+ = \bs 0 \in \mc C^+(\mc O_\Delta)$.

\smallskip

\item
If $p\t \Delta$, then $f_p=1$, \ $\tau \mf p_p=\mf p_p$, and $p \mc O_\Delta = \mf p_p^2$, and $2[\mf p_p]^+ = \bs 0\in \mc C^+(\mc O_\Delta)$.
\end{itemize}
Since apparently, $\{\mf p_p, \tau \mf p_p \} = \{\mf a \in \ms I_m(\mc O_\Delta) \colon \mf N(\mf a) = p^{f_p} \}$, it follows that the classes $F_p = \Phi_\Delta^{-1}([\mf p_p]^+)$  and $-F_p = \Phi_\Delta^{-1}(\tau[\mf p_p]^+)$ are the only classes in $\mf F_\Delta$ representing $p^{f_p}$. Moreover, $F_p=O$ if $\bigl(\frac\Delta p)=-1$, and $F_p \ne -F_p$ if and only if $\bigl(\frac \Delta p\bigr)=1$.

\medskip

Now we are ready for the main result of this section. We set $\P_m'= \{p \in \P_m \colon f_p=1\}$, and we define
\[
\vartheta \colon R_m(\Delta) \,\to\, \mc F(\mf F_\Delta) \quad \text{by}\quad \vartheta (n) = \prod_{p \in \P_m} F_p^{\mathsf v_p(n)/f_p}.
\]
Since $\mc C^+(\mc O_\Delta) = \{[\mf p_p]^+ \colon p \in \P_m'\}$, we obtain $\mf F_\Delta = \{F_p \colon p \in \P_m'\}$, and therefore $\vartheta$ is surjective.

\medskip
\begin{theorem}~ \label{3.5}

\begin{enumerate}
\item
Let $n \in R_m(\Delta)$. Then \,$n \in R_m^\circ(\Delta)\, \iff\, \vartheta(n) \in \mc B_{\pm}(\mf F_{\Delta})$.

\smallskip

\item
$\vartheta_0 = \vartheta \rest R_m^\circ(\Delta) \colon R_m^\circ(\Delta) \to \mc B_{\pm}(\mf F_\Delta)$ is a transfer homomorphism satisfying $\mathsf c(R_m^\circ(\Delta), \vartheta_0)\le 2$.
\end{enumerate}
\end{theorem}

\begin{proof}
Let $\overline \Phi_\Delta \colon \mc F(\mf F_\Delta) \to \mc F(C^+(\mc O_\Delta))$ be the unique isomorphism satisfying $\overline \Phi_\Delta \rest \mf F_\Delta = \Phi_\Delta$. Since $\Phi_\Delta (F_p) = [\mf p_p]^+$, we obtain the commutative diagram
\[
\begin{CD}
\vartheta\colon R_m(\Delta) @>>> \mc F(\mf F_\Delta)\\
@| @VV\overline\Phi_\Delta V\\
\Theta \colon N_m(K) @>>> \mc F(\mc C^+(\mc O_\Delta))
\end{CD}
\]
Since $\Phi_\Delta(-F) = \tau \Phi_\Delta(F)$  for all $F \in \mf F_\Delta$, it follows that $\overline \Phi_\Delta(\mc B_\pm(\mf F_\Delta)) = \mc B_\Gamma(\mc C^+(\mc O_\Delta))$. By  Theorems \ref{3.2} and \ref{3.4}, we obtain for $n \in R_m(\Delta) =N_m(K)$:
\[
n \in R_m^\circ(\Delta) \,\iff\, n \in N_m^+(\mc O_\Delta) \,\iff\, \Theta(n) \in \mc B_\Gamma(\mc C^+(\mc O_\Delta)) \,\iff\, \vartheta(n)\in \mc B_\pm(\mf F_\Delta) \,.
\]
The above diagram induces the diagram
\[
\begin{CD}
\vartheta_0 \colon R_m^\circ(\Delta) @>>> \mc B_\pm(\mf F_\Delta)\\
@| @VV\overline\Phi_\Delta\rest \mc B_\pm(\mf F_\Delta) V\\
\theta \colon N_m^+(\mc O_\Delta) @>>> \mc \mc B_\Gamma(\mc C^+(\mc O_\Delta)),
\end{CD}
\]
and as $\theta$ is a transfer homomorphism satisfyin $\mathsf c(N_m^+(\mc O_\Delta), \theta)$, the same is true for $\vartheta_0$.
\end{proof}

\medskip

Although in Theorem \ref{3.5} we (formally) not excluded primes $p\in\P_m\setminus \P_m'$, they really play no role for the theory. Indeed, if $n \in \N_m$ and $p \in \P_m\setminus \P_m'$, then $np^2 \in R_m(\Delta)$ if and only if $n \in R_m(\Delta)$, and $np^2 \in R_m^\circ(\Delta)$ if and only if $n \in R_m^\circ(\Delta)$.

We finally show how to modify the theory to disregard primes not in $\P_m'$. For this, we set
\[
N_m'= \mc F(\P_m'), \ \ R_m' = R_m \cap \N_m' \ \text{ and } \ R_m^{\prime \circ} = R_m^\circ \cap N_m',
\]
and we define
\[
\vartheta' \colon R_m'(\Delta) \,\to\, \mc F(\mf F_\Delta) \quad \text{by}\quad \vartheta (n) = \prod_{p \in \P_m'} F_p^{\mathsf v_p(n)}.
\]
Since $F_p=O$ for all $p\in \P_m \setminus \P_m'$, the following Corollary \ref{3.6} follows immediately from Theorem \ref{3.5}.

\smallskip
\begin{corollary}~ \label{3.6}

\begin{enumerate}
\item
Let $n \in R_m'(\Delta)$. Then \,$n \in R_m^{\prime \circ}(\Delta)\, \iff\, \vartheta'(n) \in \mc B_{\pm}(\mf F_{\Delta})$.

\smallskip

\item
$\vartheta_0' = \vartheta' \rest R_m^{\prime \circ}(\Delta) \colon R_m^{\prime \circ}(\Delta) \to \mc B_{\pm}(\mf F_{\Delta})$ is a transfer homomorphism satisfyint $\mathsf c(R_m^{\prime \circ}(\Delta), \vartheta_0)\le 2$.
\end{enumerate}
\end{corollary}

\medskip
\section{Weakly Krull and transfer Krull monoids} \label{4}
\medskip

Weakly Krull domains and weakly Krull monoids were introduced by Anderson, Mott, and Zafrullah (\cite{An-Mo-Za92}) and by Halter-Koch (\cite{HK95a}). Weakly Krull domains generalize one-dimensional noetherian domains, and a domain  is weakly Krull if and only if its monoid of nonzero elements is weakly Krull. This characterization  generalizes from domains to rings with zero-divisors (\cite[Theorem 4.4]{Ch-Oh22a}). If a domain $R$ is weakly Krull Mori, then its monoid of invertible ideals is weakly Krull Mori (\cite[Theorem 4.3]{Ge-Kh22b}). A  monoid algebra $R = D[H]$, where $D$ is a domain of characteristic zero and $H$ is a torsion-free monoid,  is weakly Krull  if and only if $D$ is a weakly Krull UMT-domain, $H$ is a weakly Krull UMT-monoid, and $\mathsf q (H)$ satisfies the ascending chain condition on cyclic subgroups (\cite[Corollary 13]{Ch-Fa-Wi22a}). Weakly factorial monoids, primary monoids (whence numerical monoids), and Krull monoids are weakly Krull monoids.  A monoid or a domain is transfer Krull if it allows a transfer homomorphism to a Krull monoid. Thus, Krull monoids are transfer Krull. For recent work on transfer Krull monoids and for a list of examples of monoids and domains that are not Krull but transfer Krull we refer to \cite{Ba-Re22a} and to the survey \cite{Ge-Zh20a}.

We recall the formal definitions of Krull, weakly Krull, and transfer Krull monoids. Each of these concepts allows a long list of equivalent conditions, and we choose the one that is most useful for us in the sequel. A monoid $H$ is said to be
\begin{itemize}
\item a\, {\it weakly Krull monoid}\, (see \cite[Corollary 22.5]{HK98}) if
      \[
      H = \bigcap_{\mathfrak p \in \mathfrak X (H)} H_{\mathfrak p} \quad \text{and} \quad \{ \mathfrak p \in \mathfrak X (H) \colon a \in \mathfrak p\} \ \text{is finite for all} \ a \in H \,;
      \]

\item a \,{\it Krull monoid}\, if $H$ is weakly Krull and $H_{\mathfrak p}$ is a discrete valuation monoid for all $\mathfrak p \in \mathfrak X (H)$ (equivalently, if there is a divisor homomorphism $\varphi \colon H \to D$, where $D$ is a free abelian monoid);

\item a {\it transfer Krull monoid} if there is a transfer homomorphism $\theta \colon H \to B$, where $B$ is a Krull monoid.
\end{itemize}
If $G$ is a finite abelian group and $G_0 \subset G$ a subset, then the inclusion $\mathcal B (G_0) \hookrightarrow \mathcal F (G_0)$ is a divisor homomorphism, and therefore $\mathcal B (G_0)$ is a Krull monoid.

If $H$ is a monoid and $S \subset H$ a submonoid, then $S$ is called {\it divisor-closed} if $a \in H$, $b \in S$, and $a \t b$ in $H$ implies that $a \in S$.

\smallskip
\begin{proposition}\label{4.1}
Let $G$ be a finite abelian group, and let $G_0\subset G$  and  $\Gamma\subset \End(G)$ be  nonempty subsets.

\smallskip

\begin{enumerate}
\item
$\{\mc B_\Gamma(G_1) \colon G_1 \subset G_0\}$ is the set of all divisor-closed submonoids of \,$\mc B_\Gamma(G_0)$.

\smallskip

\item
For $g \in G_0$ and for a subset $X \subset G_0$, let $\mf p_g = g\mc F(G_0) \cap \mc B_\Gamma(G_0)$  and
\[
\mf p_X = \bigcup_{g \in X} \mf p_g = \{ S \in \mc B_\Gamma(G_0)\colon \mathsf v_g(S) \ge 1 \ \text{ for some } \ g \in X\}.
\]
Then $s\text{-}\spec \big(\mathcal B_{\Gamma}(G_0) \big) = \{\mathfrak p_X \colon  X\subset G_0\}$, and $\mathfrak X(\mathcal B_{\Gamma}(G_0))= \{\mathfrak p_g\colon g\in G_0\}$.
\end{enumerate}
\end{proposition}

\begin{proof}
1. If $G_1 \subset G_0$, then  $\mc B_\Gamma(G_1)= \mc B_\Gamma(G_0)\cap \mc F(G_1)$ is a divisor-closed submonoid of $\mc B_\Gamma(G_0)$. Conversely, let $H \subset \mc B_\Gamma(G_0)$ be a divisor-closed submonoid, and set
\[
G_1=\bigcup_{A \in H}\supp(A).
\]
Then, obviously, $H\subset \mathcal B_{\Gamma}(G_1)$, and it remains to verify the reverse inclusion. Let $S=g_1\cdot\ldots\cdot g_{\ell}\in \mathcal B_{\Gamma}(G_1)$, where $\ell\in \N$ and $g_1,\ldots,g_{\ell}\in G_1$. Then there exist $T_1,\ldots, T_{\ell}\in H$ such that $g_i \in \supp ( T_i)$ for all $i\in [1,\ell]$.
 Therefore $T:=(T_1\cdot\ldots\cdot T_{\ell})^{\exp(G)}\in H$ is zero-sum.
 Since $S^{\exp(G)-1}\in \mathcal B_{\Gamma}(G_1)$ and $(S^{\exp(G)})^{-1}T$ is zero-sum, we have 
 $S^{-1}T= (S^{\exp(G)})^{-1}T \cdot S^{\exp(G)-1}\in \mathcal B_{\Gamma}(G_1)$ 
 and  hence $S$ divides $T$  in $\mathcal B_{\Gamma}(G_1)$. It follows from the fact that $H$ is divisor-closed in  $\mathcal B_{\Gamma}(G_1)$ that $S\in H$.

\smallskip

2. If $X \subset G_0$, then clearly $\mf p_X \in s\text{-}\spec \big(\mc B_\Gamma(G_0) \big)$. Conversely, if $\mf p \in s\text{-}\spec(\mc B_\Gamma(G_0))$, then $\mc B_\Gamma(G_0) \setminus\mf p$ is a divisor-closed submonoid of $\mc B_\Gamma(G_0)$. Thus,   $\mc B_\Gamma(G_0) \setminus \mf p = \mc B_\Gamma(G_1)$ for a subset $G_1 \subset G_0$, and therefore $\mf p = \mc B_\Gamma(G_0) \setminus \mc B_\Gamma(G_1)  = \mf p_{G_0\setminus G_1}$. Since $\mf X(\mc B_\Gamma(G_0))$ is the set of all minimal nonempty elements of $s\text{-}\spec(\mc B_\Gamma(G_0))$, we get $\mf X(\mc B_\Gamma(G_0)) = \{\mf p_g \colon g \in G_0 \}$.
\end{proof}

\smallskip
\begin{lemma}\label{4.2}
Let $G$ be a finite abelian group, and let $G_0 \subset G$ and $\Gamma \subset \End(G)$ be nonempty subsets.

\begin{enumerate}
\item
If \,$0 \in \Gamma$, then $\mc B_\Gamma(G_0) = \mc F(G_0)$.

\smallskip

\item
$\Gamma \subset \Aut(G)$ \,if and only if\, $G \cap \mc B_\Gamma(G) = \{0\}$.

\smallskip

\item
Suppose that $\{\id_G,-\id_G\} \subset \Gamma$.

\smallskip

\begin{enumerate}
\item
If \,$S \in \mc F(G_0)$ and $n \in \N$ is even, then $S^n \in \mc B_\Gamma(G_0)$.

\smallskip

\item
$\mc B_\Gamma(G_0)' = \{ S \in \mc F(G_0) \colon S^m \in \mc B_\Gamma(G_0) \ \text{ for some odd } \ m \in \N\}$. In particular, if $g \in G$ and \,$\ord(g)$ is odd, then $g \in \mc B_\Gamma(G)'$.
\end{enumerate}
\end{enumerate}
\end{lemma}

\begin{proof}
1. Obvious.

\smallskip

2. If $g \in G$, then $g \in \mc B_\Gamma (G)$ if and only if $g \in \Ker(\gamma)$ for some $\gamma \in \Gamma$. Hence
\[
G \cap \mc B_\Gamma(G) = \bigcup_{\gamma \in \Gamma} \Ker(\gamma),
\]
and the assertion follows.

\smallskip

3.\,(a)\, If $n = 2m$ for some $m \in \N$, then $S^n = S^mS^m$ and $\sigma(S^m)-\sigma(S^m) = 0$, whence $S^n \in \mc B_\Gamma(G)$.

\smallskip

(b)\, As $\mc B_\Gamma(G_0)' \subset \mc F(G_0)$, it follows that
\[
\mc B_\Gamma(G_0)' =  \{S \in \mc F(G_0) \colon S^n  \in \mc B_\Gamma(G_0) \ \text{ for all sufficiently large } \ n \in \N \}.
\]
Thus suppose that $S \in  \mc F(G_0)$ and $S^m \in\mc B_\Gamma(G_0)$ for some odd $m \in \N$. If $n \in \N$ and $n\ge 2m$, then $n = \nu m +2\mu$ for some $\nu,\,\mu \in \N_0$, and therefore $S^n = (S^m)^\nu S^{2\mu}\in \mc B_\Gamma(G_0)$. Hence $S \in \mc B_\Gamma(G_0)'$.
\end{proof}

\smallskip
\begin{proposition} \label{4.3}
Let $G$ be a finite abelian group, $G_0\subset G$ be a nonempty subset, and  $\Gamma\subset \Aut(G)$ be a subset such that $\{\id_G, -\id_G\}\subset \Gamma$.

\smallskip

\begin{enumerate}
\item
Suppose that there exists some $g\in G$ such that $\ord(g)=n \ge 3$ is odd and $\{g,-g\}\subset G_0$. Then $\mathcal B_{\Gamma}(G_0)$ is not weakly Krull.

\smallskip

\item
Suppose that there exists some $g\in G$ such that $\ord(g)=n\ge 4$ is even and $\{g,-g,\, 2g\}\subset G_0$. Then $\mathcal B_{\Gamma}(G_0)$ is not weakly Krull.

\smallskip
		
		\item $\mathcal B_{\Gamma}(G)$ is weakly Krull if and only if $G$ is an elementary $2$-group and $\Gamma=\{\id_G, -\id_G\}$.
	\end{enumerate}
 \end{proposition}

\begin{proof}
Recall from Proposition \ref{4.1} that $\mathfrak X(\mathcal B_{\Gamma}(G_0))= \{\mathfrak p_g\colon g\in G_0\}$, where $\mathfrak p_g=\{S\in\mathcal B_{\Gamma}(G_0)\colon \mathsf v_g(S)\geq 1\}$ for every $g\in G_0$.

1. Since $\{g^2, (-g)^2, g(-g), g^n, (-g)^n\} \subset \mc B_\Gamma(G_0)$, it follows that
\[
g=\frac{(g^2)^{(n+1)/2}}{g^{n}}=\frac{(-g)g\cdot \big((-g)^2 \big)^{(n-1)/2}}{(-g)^{n}}\in \bigcap_{\mathfrak p\in \mathfrak X (\mathcal B_{\Gamma}(G_0) )} \mathcal B_{\Gamma}(G_0)_{\mathfrak p}\,.
\]
But $g\not\in B_{\Gamma}(G_0)$ by Lemma \ref{4.2}.2, and therefore $B_{\Gamma}(G_0)$ is not weakly Krull.

\smallskip

2. Since $n-2$ is even, it follows that $\{(2g)g^{n-2}, (2g)(-g)^{n-2}, g^{n-2}, (-g)^{n-2} \} \subset \mc B_\Gamma(G)$, and thus
\[
2g=\frac{(2g)g^{n-2}}{g^{n-2}}=\frac{(2g)(-g)^{n-2}}{(-g)^{n-2}}\in \bigcap_{\mathfrak p\in \mathfrak X (\mathcal B_{\Gamma}(G_0) )} \mathcal B_{\Gamma}(G_0)_{\mathfrak p}\,.
\]
But $2g\not\in B_{\Gamma}(G_0)$ by Lemma \ref{4.2}.2, and therefore $B_{\Gamma}(G_0))$ is not weakly Krull.

\smallskip

3. If $G$ is an elementary $2$-group and $\Gamma=\{\id_G, -\id_G\}$, then $\mathcal B_{\Gamma}(G)=\mathcal B(G)$ is Krull and hence weakly Krull.
If $G$ is not an elementary $2$-group, then $\mathcal B_{\Gamma}(G)$ is not weakly Krull  by 1. and 2.

Now suppose $G$ is an elementary $2$-group and $\Gamma\neq\{\id_G, -\id_G\}$. Since $\id_G= -\id_G$, it follows that $G\cong C_2^r$ with $r\ge 2$, and there exist some $\tau\in \Gamma$ and $e\in G$ such that $\tau(e)\neq e$. It follows that
\[
e+\tau(e)=\frac{e^2 \big(e+\tau(e) \big)}{e^2}
=\frac{\tau(e)^2 \big(e+\tau(e) \big)}{\tau(e)^2}\in\bigcap_{\mathfrak p\in \mathfrak X(\mathcal B_{\Gamma}(G_0))} \mathcal B_{\Gamma}(G_0)_{\mathfrak p}\,.
\]
But $e+\tau(e)\not\in B_{\Gamma}(G)$ by Lemma \ref{4.2}.2, and therefore $B_{\Gamma}(G))$ is not weakly Krull.
\end{proof}

\smallskip
\begin{theorem} \label{4.4}
Let $G$ be a finite abelian group and let  $\Gamma\subset \Aut(G)$ be a subset with $\{\id_G, -\id_G\}\subset \Gamma$. Then the following statements are equivalent.

\smallskip

\begin{enumerate}
\item[(a)] $\mathcal B_{\Gamma} (G)$ is root closed.

\smallskip

\item[(b)] $\mathcal B_{\Gamma} (G)$ is Krull.

\smallskip

\item[(c)] $\mathcal B_{\Gamma} (G)$ is transfer Krull.

\smallskip

\item[(d)] $\mathcal B_{\Gamma} (G)$ is weakly Krull.

\smallskip

\item[(e)] $G$ is an elementary $2$-group and $\Gamma = \{\id_G, -\id_G \}$.
\end{enumerate}
\end{theorem}

\begin{proof}
(a) $\Longrightarrow$ (b)\,  Since $\mathcal B_{\Gamma} (G)$ is finitely generated, this follows from  \cite[Theorem 2.7.14]{Ge-HK06a}.

\smallskip

(b) $\Longrightarrow$ (c) \, This holds by the  definition of being transfer Krull.

\smallskip

(c) $\Longrightarrow$ (e)\, Let $\mathcal B_{\Gamma} (G)$ be transfer Krull and $\theta\colon B_{\Gamma}(G)\rightarrow \mathcal B(G_0)$ a transfer homomorphism, where $G_0$ is a subset of some abelian group (see \cite[Section 5]{Ge-Zh20a}). By \cite[Lemma 3.7]{B-M-O-S22}, \,$G$ is an elementary $2$-group, and we assume to the contrary that $\Gamma \ne  \{\id_G, -\id_G \}$. Since $e\tau(e)$, $e\tau(e) \big(e+\tau(e) \big)$ and $\big(e+\tau(e) \big)^2$ are  atoms of $\mathcal B_{\Gamma}(G)$ and $(e\tau(e))^2 \t \big( e \tau(e)(e+ \tau(e) \big)^2$, it follows that $\theta(e\tau(e))^2 \mid \theta(e \tau(e)(e+\tau(e))^2$, and consequently $\theta(e\tau(e))\t \theta(e \tau(e)(e+\tau(e))$. Since $\theta(e\tau(e))$ and $\theta(e\tau(e)(e+\tau(e))$ are atoms in $\mc B(G_0)$ and $\mc B(G_0)$ is reduced, we get $\theta(e\tau(e))=\theta(e\tau(e)(e+\tau(e))$. Since $\theta(e\tau(e)(e+\tau(e))^2 = \theta(e\tau(e))^2 \theta(e+\tau(e))^2$, it follows that $\theta \big( (e+\tau(e))^2 \big) =  1_{\mathcal F (G)}$, a contradiction.

\smallskip

(e) $\Longrightarrow$ (a)\, If $G$ is an elementary $2$-group and $\Gamma = \{\id_G, -\id_G \}$, then $\mathcal B_{\Gamma} (G) = \mathcal B (G)$ is Krull and hence root closed.

(b) $\Longrightarrow$ (d)\, Obvious.

\smallskip

(d) $\Longrightarrow$ (e) \, This follows by Proposition \ref{4.3}.
\end{proof}

The characterizations, provided in Theorem \ref{4.4}, are similar to what is known about the monoid of product-one sequences over (not necessarily abelian) groups (see \cite[Theorem 3.14]{Fa-Zh22a}).

\medskip
\section{C-monoids and class semigroups} \label{5}
\medskip

C-monoids are defined as submonoids of factorial monoids with finite class semigroup. We refer to \cite[Sections 2.8,\,2.9 and 3.3]{Ge-HK06a}\, for their algebraic and arithmetic theory. For the convenience of the reader, we recall the basic definitions.

Let  $F = F^{\times} \times \mathcal F (P)$ be a factorial monoid and let $H \subset F$ be a submonoid. Two elements $y, y' \in F$ are called $H$-equivalent (we write $y \sim y'$) if $y^{-1}H \cap F = {y'}^{-1} H \cap F$ or, in other words,
\[
\text{if for all} \ x \in F, \ \text{we have} \ xy \in H \quad \text{if and only if} \quad xy' \in H \,.
\]
$H$-equivalence (as introduced in \cite[Section 4]{Ge-HK04a}) is a congruence relation on $F$, and for $y \in F$, we denote by $[y]_H^F = [y]$ its congruence class. Then
\[
\mathcal C (H,F) = \{[y] \colon y \in F \} \quad \text{and} \quad \mathcal C ^* (H,F) = \{ [y] \colon y \in (F \setminus F^{\times}) \cup \{1\} \} \subset \mc C(H,F)
\]
are commutative semigroups with identity element $[1]$, \ $\mathcal C (H,F)$ is the \,{\it class semigroup}\, and \,$\mathcal C^* (H,F) $ is the \,{\it reduced class semigroup}\, of $H$ in $F$.
A monoid $H$ is said to be a {\it \C-monoid} (defined in a factorial monoid $F$) if it is a submonoid of $F$ such that $H \cap F^{\times} = H^{\times}$ and $\mathcal C^* (H, F)$ is finite.

A commutative ring is said to be a  C-ring if its monoid  of regular elements is a C-monoid. If $H$ is a C-monoid, then $H$ is a Mori monoid with nonempty conductor $(H \DP \widehat H)$, and $\widehat H$ is a Krull monoid with finite class group. Every Krull monoid with finite class group is a C-monoid, and its class semigroup coincides with its usual class group. A finitely generated monoid is a C-monoid if and only if the class group of its complete integral closure is finite (\cite[Proposition 4.8]{Ge-Ha08b}). Orders in algebraic number fields are C-domains. More generally, consider
a $v$-Marot Mori ring $R$ with nonzero conductor $\mathfrak f = (R \DP \widehat R)$. If $\mathcal C (\widehat R)$ and $\widehat{R}/\mathfrak f$ are finite, then $R$ is a C-ring (\cite[Theorem 4.8]{Ge-Ra-Re15c}), and in some special cases the converse holds (\cite{Re13a}). Apart from commutative algebra, C-monoids pop up in the theory of $C^*$-algebras (\cite{Br20a, Br21a}). A more general concept of C-domains (covering a larger class of Mori domains) is studied in \cite{Ka16b}.

The class semigroup of a C-monoid $H$ is a union of groups (i.~e., a Clifford semigroup) if and only if $H$ is seminormal (\cite{Ge-Zh19c}). For this reason, we first derive a characterization of when a monoid $\mathcal B_{\Gamma} (G)$ is seminormal (Theorem \ref{5.2}). Then, we establish an explicit description of the  class semigroup of $\mathcal B_{\pm} (G)$ in the seminormal case (Theorem \ref{5.5}). So far, explicit descriptions of  class semigroups (which are not groups) have been found only for  monoids of product-one sequences over finite non-abelian groups (\cite[Section 4]{Oh20a}).

\smallskip

The following Theorem \ref{5.1}\, is due to \cite[Theorem 3.10]{B-M-O-S22}. We provide it with a simple proof using \cite[Proposition 2.6]{Cz-Do-Ge16a}.

\smallskip
\begin{theorem} \label{5.1}
Let $G$ be a finite abelian group, and let $G_0 \subset G$ and $\Gamma \subset \End (G)$ be nonempty subsets. Then $\mathcal B_{\Gamma} (G_0)$ is a \C-monoid defined in $\mathcal F (G_0)$.
\end{theorem}

\begin{proof}
Since $\mathcal B_{\Gamma} (G_0)$ is finitely generated and $S^{\exp (G)} \in \mathcal B_{\Gamma} (G_0)$ for every $S \in \mathcal F (G_0)$, it follows by \cite[Proposition 2.6]{Cz-Do-Ge16a} that $\mc B_\Gamma(G_0)$ is a C-monoid.
\end{proof}

We need a simple technical lemma.

\smallskip
\begin{lemma}\label{lemma}
Let $G = C_{2^{t_1}}\oplus \ldots\oplus C_{2^{t_r}}$ be a finite abelian $2$-group,  where $r,t_1,\ldots,t_r\in \N$ with $t_1<\ldots<t_r$, and let $(e_1,\ldots,e_r)$ be a basis of $G$ with $\ord(e_i)=2^{t_i}$ for all $i\in [1,r]$.
If $m \in [1,r]$, \ $k \in \N$ odd, and $\alpha \in \langle e_1, \ldots, e_{m-1}, 2e_m, e_{m+1}, \ldots, e_r\rangle$ with $\ord(\alpha)\le 2^{t_m}$, then $(e_1,\ldots, e_{m-1}, ke_m+\alpha,e_{m+1}, \ldots, e_r)$ is a basis of \,$G$.
\end{lemma}

\begin{proof}
Let $\alpha = a_1e_1+ \ldots + a_re_r$, where $a_1, \ldots, a_r \in \Z$ and $2 \t a_m$. If $k'\in \Z$ is such that $k'(k+a_m) \equiv 1 \mmod 2^{t_m}$, then
\[
e_m = k'(ke_m +\alpha) - \sum_{\substack{j=1\\ j\ne m}}^r k'a_je_j, \quad \text{and thus} \quad G = \langle e_1, \ldots, e_{m-1}, ke_m+\alpha, e_{m+1}, \ldots, e_r\rangle.
\]
To prove independence, let $b_1, \ldots, b_r \in \Z$ such that
\[
0 = b_m(ke_m+\alpha) + \sum_{\substack{j=1\\ j\ne m}}^r b_je_j  = b_m(k+a_m)e_m +\sum_{\substack{j=1\\ j\ne m}}^r (b_ma_j+b_j)e_j.
\]
It follows that $b_m(k+a_m) e_m=0$, hence $2^{t_m} \t b_m$, and thus $b_m(ke_m+\alpha)=0$, since $\,\ord(\alpha) \le 2^{t_m}$. Now the independence of $(e_1, \ldots, e_r)$ implies $b_je_j=0$ for all $j \in [1,r]\setminus \{m\}$.
\end{proof}

Our next result offers  a characterization of when $\mathcal B_{\Gamma} (G)$ is seminormal and characterizes its complete integral closure.

\smallskip
\begin{theorem}\label{5.2}
Let $G$ be a finite abelian group, $|G|>1$, and let $\Gamma\subset \Aut(G)$ be a subset such that $\{\id_G, -\id_G\}\subset \Gamma$.
\begin{enumerate}

\smallskip

\item
If $\mathcal B_{\Gamma}(G)$ is seminormal, then $\exp(G)$ is a power of \,$2$.

\smallskip
		
\item
$\mathcal B_\pm(G)$ is seminormal if and only if $\exp(G)\t 4$. If this holds, then
\[\qquad    \quad
\widehat{\mathcal B_\pm(G)}=\{S\in \mathcal F(G) \colon \sigma(S)\in 2G\}.
\]

\smallskip
		
\item
Let $\Gamma=\Aut(G)$.

\smallskip

\begin{enumerate}
\item
$\mathcal B_{\Gamma}(G)$ is seminormal if and only if $G\cong C_{2^{t_1}}\oplus \ldots\oplus C_{2^{t_r}}$, where $r,\,t_1,\ldots,t_r\in \N$ and $t_1<\ldots<t_r$.

\smallskip

\item
Let $G\cong C_{2^{t_1}}\oplus \ldots\oplus C_{2^{t_r}}$, where $r,\,t_1,\ldots,t_r\in \N$ and $t_1<\ldots<t_r$. For $S = g_1 \cdot \ldots \cdot g_\ell \in \mc F(G)$, where $\ell\in \N$ and $g_1,\ldots, g_\ell \in G$, we set $\mathsf n(S) = |\{j \in [1,\ell]\colon \ord(g_j)=2^{t_r}\}|$. Then
\[\qquad
\widehat{\mathcal B_\Gamma(G)}=\{S\in \mathcal F(G)\colon \mathsf n(S) \ \text{ is even }\}.
\]
\end{enumerate}
\end{enumerate}
\end{theorem}

\begin{proof}
Recall that $\mc B_\Gamma(G)' \subset \widehat{\mc B_\Gamma (G)} \subset \widehat{\mc F(G)} = \mc F(G)$.

\smallskip

1.	Suppose that $\exp(G)$ is not a power of $2$. Then there exists some $g\in G$ such that $\ord(g) =n \ge 3$ is odd, and by Lemma \ref{4.2} it follows that $g \in \mc B_\Gamma(G)' \setminus \mc B_\Gamma(G)$. Hence $\mc B_\Gamma(G)$ is not seminormal,
	
\smallskip
	
2. If $\exp(G)$ is not a power of $2$, then $\mathcal B_\pm(G)$ is not seminormal by 1. If $\exp(G)=2^r$ with $r\ge 3$, then there exists some $g\in G$ such that $\ord(g)=8$, and we set $S=g(5g) \in \mc F(G)$. It follows that $S^2\in \mathcal B_\pm(G)$ and $S^3 = \big( g^3(5g) \big) (5g)^2\in \mathcal B_\pm(G)$, whence $S^n \in\mc B_\pm(G)$ for all $n \ge 2$ and thus $S\in (\mathcal B_\pm(G))'$. Since $\{\pm g\pm 5g\}=\{4g, \pm 2g\}$, we obtain $S\notin \mathcal B_\pm(G)$, and therefore $\mathcal B_\pm(G)$ is not seminormal.
	
Suppose now that $\exp(G)\t 4$. Let $S\in (\mathcal B_\pm(G))'$ and $t \in \N$ be odd  such that $S^t\in \mathcal B_\pm(G)$. We shall prove that $S \in \mc B_\pm(G)$. Let $S=g_1 \cdot \ldots \cdot g_\ell$, where $\ell\in \N_0$, \ $g_1,\ldots,g_\ell\in G$ and $\ord(g_i)\t 4$ for all $i \in [1,\ell]$. Then $S^t = S'S''$ for some $S',\,S'' \in \mc F(G)$ such that $\sigma(S')-\sigma(S'')=0$. We set
\[
S'= \prod_{i=1}^\ell g_i^{x_i} \quad \text{and} \quad S''= \prod_{i=1}^\ell g_i^{y_i}
\]
where $x_i,\,y_i \in \N_0$ and $x_i+y_i = t$ for all $i \in [1,\ell]$. Since $x_i-y_i = t -2y_i \equiv 1 \mmod 2$, there exists some $e_i \in\{\pm 1\}$ such that $x_i-y_i \equiv e_i \mmod 4$ for all $i \in [1,\ell]$. Now we obtain
\[
0 = \sigma(S')-\sigma(S'') = \sum_{i=1}^\ell (x_i-y_i)g_i = \sum_{i=1}^\ell e_ig_i,
\]
which shows that $S \in \mc B_\pm(G)$.

\smallskip

It remains to consider $\wh{\mc B_\pm(G)}$. If $\exp(G) \t 2$, then $\mc B_\pm(G) = \mc B(G)$ is a Krull monoid, and there is nothing to do. Thus let \,$\exp(G) = 4$, say $G = C_2^r\oplus C_4^t$ for some $r \in \N_0$ and $t \in \N$.
 and let $(e_1,\ldots,e_r,f_1,\ldots,f_t)$ be basis of $G$ such that $\,\ord(e_i)=2$ for all $i \in [1,r]$, and \,$\ord(f_j)=4$ for all $j \in [1,t]$. Then $(2f_1,\ldots,2f_t)$ is a basis of $2G$, we set $S_0=f_1^2\cdot\ldots\cdot f_t^2$ and obtain
\begin{equation} \label{5.2.1}
\sigma_\pm(S_0) = \Bigl\{\sum_{j=1}^t (\varepsilon_j-\varepsilon'_j)f_j \colon \varepsilon_j,\,\varepsilon_j' \in \{\pm 1\} \Bigr\} = 2G.
\end{equation}
Note, if $S = g_1 \cdot \ldots \cdot g_{\ell} \in \mathcal F (G)$, then
\begin{equation} \label{5.2.2}
\begin{aligned}
\sigma_{\pm} (S) & = \sum_{\substack{i=1 \\ \ord (g_i) \le 2}}^{\ell} g_i + \sum_{\substack{i=1\\ \ord(g_i)= 4}}^\ell \{g_i,-g_i\} = \sum_{\substack{i=1 \\ \ord (g_i) \le 2}}^{\ell} g_i + \sum_{\substack{i=1 \\ \ord (g_i) = 4}}^{\ell} \big(  g_i + \{0, 2g_i\} \big) \\
 & = \sigma (S) + \sum_{\substack{i=1\\ \ord(g_i)= 4}}^\ell \{0, 2g_i\} \subset \sigma (S) + 2G \,.
\end{aligned}
\end{equation}
Now let $S \in \mc F (G)$ be such that $\sigma(S)\in 2G$. Then, by \eqref{5.2.1} and \eqref{5.2.2}, it follows that
\[
\begin{aligned}
\sigma_\pm (S_0S) & = \sigma_\pm(S_0) + \sigma_\pm(S) = 2G +   \sigma_\pm(S) =2G.
\end{aligned}
\]
In particular, $0\in \sigma_{\pm}(S_0S)$, and therefore $S_0S\in \mathcal B_\pm(G)$.
Since $S^2\in \mathcal B_\pm(G)$, it follows that $S_0S^k\in \mathcal B_\pm(G)$ for all $k\in \N$, and therefore $S\in \widehat{\mathcal B_{\pm}(G)}$.

Conversely, let $S\in \widehat{\mathcal B_{\pm}(G)}$, and let $S_1\in \mathcal B_{\pm}(G)$ such that $S_1S\in \mathcal B_{\pm}(G)$. Then, \eqref{5.2.2} implies that $0 \in \sigma_{\pm} (S_1) \subset \sigma (S_1) + 2G$ and $0 \in \sigma_{\pm} (S_1S) \subset \sigma (S_1S) + 2G$, whence $\sigma(S_1),\,\sigma(S_1S)\in 2G$, and consequently  $\sigma(S)=\sigma(S_1S)-\sigma(S_1)\in 2G$.

\smallskip

3.(a)\, If $\mathcal B_{\Gamma}(G)$ is  seminormal, then $\exp(G)$ is a power of $2$ by 1., and we set
\[
G = C_{2^{t_1}}^{s_1}\oplus \ldots\oplus C_{2^{t_r}}^{s_r} \,, \quad \text{ where} \quad  r, t_1,\ldots, t_r, s_1,\ldots,s_r\in \N \ \text{ and } \  t_1<\ldots<t_r \,.
\]
Assume to the contrary that  $s_i\ge 2$ for some  $i\in [1,r]$. Then there exist independent elements $e_1,e_2\in G$ such that $\ord(e_1)=\ord(e_2)=2^{t_i}$ and $\tau_1,\tau_2\in \Gamma$ such that $\tau_1(e_1)=e_2$ and $\tau_2(e_1)=-e_1-e_2$.
Then $e_1+\tau_1(e_1)+\tau_2(e_1)=0$, hence $e_1^3\in \mathcal B_{\Gamma}(G)$, and therefore $e_1\in (\mathcal B_{\Gamma}(G))'\setminus \mathcal B_{\Gamma}(G)$ by Lemma \ref{4.2}, a contradiction.

\smallskip
As to the converse, we suppose that
\[
G = C_{2^{t_1}}\oplus \ldots\oplus C_{2^{t_r}} \,, \quad \text{ where } \quad r,t_1,\ldots,t_r\in \N \ \text{ and } \ t_1<\ldots<t_r ,
\]
and we assert that $\mathcal B_{\Gamma}(G)$ is seminormal. Let $(e_1,\ldots,e_r)$ be a basis of $G$ and $\ord(e_i)=2^{t_i}$ for all $i\in [1,r]$. For $g \in G$ and $i \in[1,r]$, let $[g]_i \in [0, 2^{t_i}-1]$ be the uniquely determined integer such that \,$g = [g]_1e_1 +\ldots+ [g]_re_r$, and hence we have $\ord(g) = \max\{\ord([g]_ie_i \colon i \in [1,r]\}$.  We denote by $\mathsf v_2 \colon \Q \to \Z \cup \{\infty\}$ the $2$-adic valuation. Hence, if $b \in \Z$, then $bg =0$ if and only if $b[g]_ie_i=0$ for all $i \in[1,r]$, if and only if $2^{\mathsf v_2(b)}g =0$ (under the convention that $2^\infty g = \infty g =0$).

\smallskip

For any sequence $S = g_1\cdot \ldots g_\ell \in \mc F(G)$, where $\ell\in \N$ and $g_1, \ldots, g_\ell \in G$, we define
\[
\mathsf d(S) = \min \{\mathsf d_1(S), \ldots, \mathsf d_r(S)\}, \quad \text{where}\quad \mathsf d_i(S) =\min\{\mathsf v_2([g_j]_i )\colon  j \in[1,\ell]\bigr\} \ \text{ for all } \ i \in [1,r],
\]
\[
\mathsf m(S) = \max\{ i \in [1,r]\colon \mathsf   \mathsf d_i(S) =\mathsf d(S)\},
\]
\[
\mc N(S) = \{j \in [1,\ell] \colon \mathsf v_2([g_j]_{\mathsf m(S)}) = \mathsf d(S)\},\quad \text{and}
\]
\[
\mc I(S) = \{i \in [1,r] \colon t_i-\mathsf d_i(S) >t_{\mathsf m(S)}-\mathsf d(S)\} .
\]
By definition, $\mathsf d_{\mathsf m(S)}(S) = \mathsf d(S)$, and therefore $\mathsf m(S) \notin \mc I(S)$. If $i \in [1,\mathsf m(S)]$, then $t_i-\mathsf d_i(S) \le t_{\mathsf m(S)} - \mathsf d(S)$, hence $i \notin \mc I(S)$, and consequently $\mc I(S) \subset [\mathsf m(S) +1,r]$. For every $p \in \N$, it follows from the very definition that $\mathsf d_i(S^p) = \mathsf d_i(S)$ for all $i \in [1,r]$, \ $\mathsf d(S^p) = \mathsf d(S)$, \ $\mathsf m(S^p) = \mathsf m(S)$, \ $|\mc N(S^p)| = p|\mc N(S)|$, and $\mc I(S^p) = \mc I(S)$.
Moreover, if $S = 0^\ell$ (equivalently, $\supp(S) = \{0\}$), then $\mathsf d_i(S) = \mathsf d(S) = \infty$ for all $i \in [1,r]$, and $\mc I(S) = \emptyset$.

\smallskip

Starting with a sequence $S = g_1\cdot \ldots g_\ell \in \mc F(G)$, where $\ell\in
\N$ and $g_1, \ldots, g_\ell \in G$, we define recursively a sequence $(S_\lambda)
_{\lambda \ge 1}$ in $\mc F(G)$ by $S_1=S$ and, for $\lambda \ge 2$,
\[
S_\lambda  = \prod_{j=1}^\ell \Bigl(\sum_{i \in \mc I(S_{\lambda-1})}[g_j]_ie_i \Bigr)= \prod_{j=1}^\ell g_j^{(\lambda)}, \ \text{ where } \ [g_j^{(\lambda)}]_i = \begin {cases}
 \,[g_j]_i &\text{ if } \ i \in\mc I(S_{\lambda-1}),\\ \ 0&\text{ if } \ i \notin\mc I(S_{\lambda-1}) \,.
\end{cases}
\]
Consequently,
\[
\mathsf d_i(S_\lambda) = \begin {cases}\,\mathsf d_i(S) &\text{if } \ i \in \mc I(S_{\lambda -1}),\\
\ \infty &\text{if } \ i \notin \mc I(S_{\lambda -1})
\end{cases}\quad \text{and } \ \mathsf d(S_\lambda) \ge \mathsf d(S)
\quad \text{for all } \ i \in [1,r] \ \text{ and } \ \lambda \ge 2.
\]
Apparently, $S = 0^\ell$ implies $S_\lambda = 0^\ell$ for all $\lambda \in \N$.
We continue with the following assertion.

\smallskip

{\bf A0.} Let $S = g_1\cdot \ldots g_\ell \in \mc F(G)$, where $\ell\in \N$ and
$g_1, \ldots, g_\ell \in G$, and let $\lambda \in \N$ with $\mc I(S_{\lambda}) \ne \emptyset$. Then $\mc I(S_{\lambda+1}) \subsetneq \mc I(S_\lambda)$ and $\mathsf m (S_{\lambda + 1}) \in \mathcal I (S_{\lambda}) \setminus \mathcal I (S_{\lambda+1})$.

\smallskip

{\it Proof of}\, {\bf A0}. If $i \in [1,r] \setminus \mc I(S_\lambda)$, then $[g_j]_i =0$ for all $j \in [1,\ell]$, hence $\mathsf d_i(S_{\lambda+1}) = \infty$ and thus $i \notin \mc I(S_{\lambda+1})$. It follows that $\mc I(S_{\lambda+1})\subset \mc I(S_\lambda)$, and as $\mathsf m(S_{\lambda+1}) \notin \mc I(S_{\lambda +1})$, it suffices to observe that $\mathsf m(S_{\lambda+1})\in \mc I(S_\lambda)$. Indeed, if $\mathsf m(S_{\lambda+1}) \notin \mc I(S_\lambda)$, then $\infty = \mathsf d_{\mathsf m(S_{\lambda +1}}) = \mathsf d(S_\lambda)$, hence $S_\lambda =0^\ell$ and $\mc I(S_\lambda) = \emptyset$, a contradiction.
\qed[{\bf A0.}]

\smallskip

{\bf A0}\, justifies the following definition. If $S \in \mc F(G)$, \ $|S|= \ell \ge 1$ and $S \ne 0^\ell$, then we set
\[
\mathsf t(S) = \min\{ \lambda \in \N \colon \mc I(S_\lambda) = \emptyset\}.
\]
We proceed with the following three claims.

\smallskip

{\bf A1.}  Let $\tau\in \Aut(G)$ and $s\in [1,r]$. Then  $2 \nmid [\tau(e_s)]_s$, and $2^{t_i-t_{s}} \t [\tau(e_s)]_i$ for every $i\in [s+1,r]$.

{\it Proof of {\bf A1.}} Since \,$\ord(\tau(e_s)) = \ord(e_s) = 2^{t_s}$, it follows that $2^{t_s}[\tau(e_s)]_i e_i = 0$ for all $i \in [1,r]$, and therefore $2^{t_i-t_s}\mid[\tau(e_s)]_i$ for all $i \in [s+1,r]$. Assume to the contrary  that $2 \t [\tau(e_s)]_s$. Let $\alpha=[\tau (e_s)]_1 e_1+\ldots+[\tau (e_s)]_{s-1}e_{s-1}$	 and $\beta=\frac{1}{2}([\tau (e_s)]_{s}e_{s}+\ldots+[\tau (e_s)]_re_r)$.
Then $e_{s}=\tau^{-1}(\alpha+2\beta)=\tau^{-1}(\alpha)+2\tau^{-1}(\beta)$. Since $\ord(\alpha)<2^{t_{s}}$, we obtain that $[\tau^{-1}(\alpha)]_s$ is even and hence $1=[e_{s}]_s= [\tau^{-1}(\alpha)]_s+2[\tau^{-1}(\beta)]_s$ is even, a contradiction.
\qed({\bf A1})
	
\smallskip

{\bf A2.} Let $S = g_1\cdot \ldots g_\ell \in \mc F(G)$, where $\ell\in \N$, \
$g_1, \ldots, g_\ell \in G$, and $S \ne 0^\ell$, and let $\tau\in \Aut(G)$, \ $\lambda
\in [ 1,\mathsf t(S)]$, and $j \in [1,\ell]$. We set $d = \mathsf d(S_\lambda)$ and $m = \mathsf m(S_\lambda)$. Then we have

\begin{itemize}
\item[(a)] $\ord([g_j]_ie_i )< \ord (2^de_m)$ for all $i \in [1,m-1]$, and $2^{d+1} \t [g_j]_i$ for all $i \in [m+1,r]$.

\item[(b)] $[g_j]_{m} \equiv [\tau(g_j)]_{m}\mmod 2^{d+1}$.
\end{itemize}

{\it Proof of {\bf A2.}}

(a) If $i \in [m+1,r]$, then
$2^{d+1} \t [g_j]_i$ by the very definition of $\mathsf m(S_\lambda)$. Thus, let $i \in [1,m-1]$. If $\lambda =1$, or if $\lambda \ge 2$ and $i \in \mc I(S_{\lambda -1})$, then $2^d \t [g_j]_i$, and therefore \,$\ord([g_j]_ie_i) \le \ord(2^de_i)< \ord(2^de_m)$.
If $\lambda \ge 2$ and $i \notin \mc I(S_{\lambda -1})$, then
\[
\ord([g_j]_ie_i) = 2^{t_i- \mathsf v_2([g_j]_i)} \overset{(i)}{\le} 2^{t_i-\mathsf d_i(S_{\lambda-1})} \overset{(ii)}{\le} 2^{t_{\mathsf m(S_{\lambda-1})}-\mathsf d(S_{\lambda-1})} \overset{(iii)}{\le} 2^{t_m-d} = \ord(2^d e_m) \,;
\]
note that (i) follows from the definition of $\mathsf d_i (S_{\lambda-1})$, (ii) follows from the definition of $\mathcal I (S_{\lambda-1})$, and (iii) holds because $m = \mathsf m (S_{\lambda}) \in \mathcal I (S_{\lambda-1})$ by {\bf A0}.

(b) By (a), we obtain a decomposition $g_j = \alpha + [g_j]_me_m + 2^{d+1}\beta$, where $\alpha \in \langle e_1, \ldots, e_{m-1}\rangle$, \ $\ord(\alpha)< \ord(2^de_m)$ and $\beta \in \langle e_{m+1}, \ldots, e_r\rangle$. From this, we obtain
\[
[\tau(g_j)]_m = [ \tau(\alpha)]_m + [g_j]_m[\tau(e_m)]_m + 2^{d+1}[\tau(\beta)]_m.
\]
Since \,$\ord(\alpha)< \ord(2^de_m)$, it follows that $\ord([\tau(\alpha)]_me_m) \le \ord(\tau(\alpha)) = \ord(\alpha) \le \ord(2^{d-1}e_m)$, hence $2^{d+1} \t [\tau(\alpha)]_m$, and we arrive at the conguence $[\tau(g_j)]_m \equiv [g_j]_m [\tau(e_m)]_m \mmod 2^{d+1}$. Since $2 \nmid [\tau(e_m)]_m$ by \,{\bf A1}\, and $2^d \t [g_j]_m$, it follows that $[g_j]_m [\tau(e_m)]_m \equiv [g_j]_m \mmod 2^{d+1}$, and we are done. \qed[{\bf A2}]

\medskip

{\bf A3.} Let $S \in \mc F(G)$ with $|S|=\ell \ge 1$ and $S \ne 0^\ell$. Then $S\in \mathcal B_{\Gamma}(G)$ if and only if $|\mc N(S_\lambda)|$ is even for all  $\lambda \in [1,\mathsf t(S)]$.

{\it Proof of {\bf A3.}} Let first $S = g_1 \cdot \ldots \cdot g_{\ell} \in \mathcal B_{\Gamma}(G)$, where $g_1, \ldots, g_{\ell} \in G$, and let $\tau_1, \ldots, \tau_\ell \in \Aut(G)$ be such that $\tau_1(g_1) + \ldots + \tau_\ell (g_\ell)=0$. Let $\lambda \in [1,\mathsf t(S)]$, \ $d = \mathsf d(S_\lambda)$ and $m = \mathsf m(S_\lambda)$. By the very definition of $\mc N(S_\lambda)$, we obtain
\[
2^{-d}[g_j]_m \equiv \begin {cases}\,1\mmod 2 &\text{if } \ j \in \mc N(S_\lambda),\\
\,0 \mmod 2 &\text{if } \ j \notin \mc N(S_\lambda),
\end{cases}\quad \text{whence} \quad [g_j]_m \equiv \begin {cases}\,2^d\mmod 2^{d+1} &\text{if } \ j \in \mc N(S_\lambda),\\
\ 0 \mmod 2^{d+1} &\text{if } \ j \notin \mc N(S_\lambda),
\end{cases}
\]
and by {\bf A2} it follows that
\[
0 = \Bigl[\sum_{j=1}^\ell \tau(g_j)\Bigr]_m= \sum_{j=1}^\ell [\tau(g_j)]_m \equiv \sum_{j=1}^\ell [g_j]_m \equiv 2^d|\mc N(S_\lambda)|\negthickspace \mmod 2^{d+1}
\]
and therefore $|\mc N(S_\lambda)|\equiv 0\mmod 2$.

For the converse, we proceed by induction on $\mathsf t(S)$ over all sequences $S \in \mc F(G)$ with $|S| =\ell \ge 1$ and $S \ne 0^\ell$. Let $S = g_1\cdot \ldots \cdot g_\ell$ be such a sequence, and suppose that $n_\lambda=|\mc N(S_\lambda)|\equiv 0 \mmod 2$ for all $\lambda \in [1,\mathsf t(S)]$.

Suppose $\mathsf t(S)=1$. Let $d_i=\mathsf d_i(S)$ for all $i \in [1,r]$, \ $d=\mathsf d(S)$, \  $m=\mathsf m(S)$, and $2n=|\mc N(S)|$. After renumbering if necessary, we may assume that $2^{-d}[g_j]_m$ is odd for all $j\in [1,2n]$ and $2^{-d}[g_j]_m$ is even for all $j\in [2n+1,\ell]$. For $j \in [1,\ell]$, we set
\[
\alpha_j = 2^{-d}(g_j - [g_j]_me_m) =  \sum_{\substack{j=1\\ j\ne m}}^r 2^{-d}[g_j]_ie_i \in \langle e_1, \ldots, e_{m-1},e_{m+1}, \ldots, e_r\rangle.
\]
Since $\mc I(S)=\emptyset$, we obtain $\ord([g_j]_ie_i) \le 2^{t_i-d_i} \le 2^{t_m-d}$, whence $\ord(2^{-d}[g_j]_ie_i) \le 2^{t_m}$ for all $i \in [1,r]$ and therefore $\ord(\alpha_j) \le 2^{t_m}$. For every odd $k \in\N$, the tuple $(e_1,\ldots, e_{m-1}, ke_m+\alpha_j, e_{m+1}, \ldots, e_r)$ is a basis of $G$ by Lemma \ref{lemma}. Hence there exists some $\tau \in \Aut(G)$ such that $\tau(ke_m+\alpha_j)=e_m$. If $k= 2^{-d}[g_j]_m$ for $j \in [1,2n-1]$, then $ke_m+\alpha_j= 2^{-d}g_j$. Next we set $\alpha = \alpha_{2n} + \ldots + \alpha_\ell$. Then again $\ord(\alpha)\le 2^{t_m}$, \ $k= 2^{-d}([g_{2n}]_m+\ldots + [g_{\ell}]_{2n})$ is odd, $ke_m+\alpha = 2^{-d}(g_{2n}+ \ldots+g_\ell)$, and there exists some $\tau \in \Aut(G)$ such that $\tau(ke_m+\alpha)=e_m$. By these arguments, there exist
$\tau_1,\ldots, \tau_{2n}\in \Aut(G)$ such that
	\begin{align*}
		\tau_{2j-1}(g_j)&=2^d\tau_{2j-1}(2^{-d}g_j)=2^de_m \text{ for all }j\in [1,n]\,,\\
		\tau_{2j}(g_j)&=2^d\tau_{2j}(2^{-d}g_j)=-2^de_m \text{ for all }j\in [1,n-1]\,,\\
		\text{and } \qquad\tau_{2n}(g_{2n}+\ldots+g_{\ell})&=2^d\tau_{2n}(2^{-d}(g_{2n}+\ldots+g_{\ell}))=-2^de_m\,,
\end{align*}
	
It follows that
\[
\sum_{j=1}^n \tau_{2j-1}(g_j) +\sum_{j=1}^{n-1} \tau_{2j}(g_{2j}) +
\sum_{j=2n}^\ell \tau_{2n}(g_j) =0,
\]
and therefore $S = g_1\cdot\ldots \cdot g_\ell\in \mathcal B_{\Gamma}(G)$.

Suppose that $\mathsf t(S)>1$ and assume that the assertion holds for $\mathsf t(S) -1$. Then $\mathcal I(S)\neq \emptyset$. Let $I_1=[1,r]\setminus \mathcal I(S)$ and $I_2=\mc I(S)$. Then $G=G_1\oplus G_2$, where $G_1=\langle e_i\colon i\in I_1\rangle$ and $G_2=\langle e_i\colon i\in I_2\rangle$. For $i \in \{1,2\}$ let $\varphi_i\colon G\rightarrow G_i$ be the projection. Then $\mathsf t(\varphi_1(S))=1$ and $\mathsf t(\varphi_2(S))=\mathsf t(S)-1$. By the induction hypothesis,  there exist $\mu_1,\ldots,\mu_{\ell}\in \Aut(G_1)$ and $\nu_1,\ldots,\nu_{\ell}\in \Aut(G_2)$ such that
\[
\sum_{j=1}^\ell \mu_j(\varphi_1(g_j)) = \sum_{j=1}^\ell \nu_j(\varphi_2(g_j))=0.
\]
For $j \in [1,\ell]$, let $\tau_j\in \Aut(G)$ such that $\tau_j\mid_{G_1}=\mu_j$ and $\tau_j\mid_{G_2}=\nu_j$. Then $\tau_j(g_j)=\mu_j(\varphi_1(g_j))+\mu_j(\varphi_2(g_j))$ for all $i\in [1,\ell]$, whence $\tau_1(g_1)+\ldots+ \tau_{\ell}(g_{\ell}) =0$ and therefore $S\in \mathcal B_{\Gamma}(G)$.	\qed({\bf A3})

\medskip

To complete the proof that $\mathcal B_{\Gamma} (G)$ is seminormal, let $S\in (\mathcal B_{\Gamma}(G))' \subset \mathcal F (G)$ be given. We need to   prove that $S\in \mathcal B_{\Gamma}(G)$. By Lemma \ref{4.2}, there exists an odd $k\in \N$ such that $S^k\in \mathcal B_{\Gamma}(G)$.
By {\bf A3}, \ $|\mc N(S_\lambda^k)|=k|\mc N(S_\lambda)|$ is even for all $\lambda\in [1,\mathsf t(S^k)]=[1,\mathsf t(S)]$, whence $|\mc N(S_\lambda)|$ is even for all $\lambda\in [1,\mathsf t(S)]$, and, again by {\bf A3} we obtain $S\in \mathcal B_{\Gamma}(G)$.

\smallskip

(b)\, We start with some preliminary considerations. If $g \in G$, then \,$\ord(g) = 2^{t_r}$ if and only if $2 \nmid [g]_r$.

Let $S \in \mc F(G)$ and $\mathsf n(S) >0$. Then $\mathsf d_r(S) = \mathsf d(S) =
0$, \ $\mathsf m(S) =r$, and $|\mc N(S)| = \mathsf n(S)$. Since $t_r -\mathsf d(S) \ge t_i -\mathsf d_i(S)$ for all $i \in [1,r]$, it follows that $\mc I(S) = \emptyset$ and $\mathsf t(S)=1$.

Now we can do the actual proof. Let $S \in \mathcal F (G)$ be such that  $\mathsf n(S)$ is even. Since $\mathsf m (e_r^2S)=r$, we obtain that $\mathcal I (e_r^2S)=\emptyset$, whence $\mathsf t(e_r^2S)= 1$. Since $\mathsf n(e_r^2S)= 2+\mathsf n(S)$ is even and  $\mathsf t(e_r^2S)= 1$, \,{\bf A3}\, implies that $e_r^2S \in \mc B_\Gamma(G)$. Since $S^2 \in \mathcal B_{\Gamma} (G)$, it follows that $e_r^2S^k  \in \mathcal B_{\Gamma} (G)$ for all $k \in \N$,  and consequently $S \in \widehat{\mc B_\Gamma(G)}$. Conversely, if $S \in \widehat{\mc B_\Gamma(G)}$, then there exists some $S_0 \in \mc B_\Gamma(G)$ such that $SS_0 \in \mc B_\Gamma(G)$. Then $\mathsf n(SS_0) = \mathsf n(S) +\mathsf n(S_0)$ and $\mathsf n(S_0)$ are both even by {\bf A3}, and thus $\mathsf n(S)$ is even.
\end{proof}

\smallskip
We continue with the preparation for what is needed to describe the class semigroup of  seminormal monoids of weighted zero-sum sequences.
Let $G$ be a finite abelian group and let $\Gamma\subset \Aut(G)$  be a  subgroup. Then $\mathcal B (G) \subset \mathcal B_{\Gamma} (G) \subset \mathcal F (G)$. The monoid $\mathcal B (G)$ is a Krull monoid and a C-monoid. It is factorial if and only if $|G| \le 2$ and if $|G| \ge 3$, then $\mathcal C ( \mathcal B (G), \mathcal F (G) ) =  \mathcal C ( \mathcal B (G) ) \cong G$, where $\mathcal C \big( \mathcal B (G) \big)$ is the (usual) class group of the Krull monoid.
The monoid $\mathcal B_{\Gamma} (G)$ is a C-monoid by Proposition \ref{5.1}. Since $\mathcal F (G)$ is reduced, we have \,$\mathcal C^* ( \mathcal B_{\Gamma} (G), \mathcal F (G)) = \mathcal C( \mathcal B_{\Gamma} (G), \mathcal F (G) )$, and, by \cite[Theorem 2.9.11]{Ge-HK06a}, there is an epimorphism\, $\mathcal C ( \mathcal B_{\Gamma} (G), \mathcal F (G) \  \rightarrow \ \mathcal C ( \widehat{\mathcal B_{\Gamma} (G)})$.

\smallskip

We refer to \cite{Gr01} for the basics on commutative semigroups. Let $\mathcal C$ be an additively written, finite, commutative semigroup with zero element. Let $\mathsf E (\mathcal C)$ be the set of idempotents of $\mc C$, endowed
with the Rees order, defined by $e \le f$ if $e+f=e$. If $\mathsf E (\mathcal C) = \{0=e_0, \ldots, e_n \}$, then  $0=e_0$ is the largest and $e_0 + \ldots + e_n$ is the smallest element of $\mathsf E(\mc C)$ in the Rees order. For every $e\in \mathsf E (\mathcal C)$, we denote by $\mathcal C_e$  the set of all $x \in \mathcal C$ such that $x+e=x$ and $x+y=e$ for some $y \in \mathcal C$. Then $\mathcal C_e$ is a group with identity element $e$, called the {\it constituent group} of $e$. In particular, $\mc C_0 = \mc C^\times$ is the group of invertible elements of $\mc C$. If $e,\,f \in\mathsf E(\mc C)$ and $e \ne f$, then $\mc C_e \cap \mc C_f = \emptyset$. Every subgroup $G$ of $\mc C$ is contained in a unique constituent group $\mc C_e$. This shows that $\mathcal C$ is a union of groups if and only if
\begin{equation}
\mathcal C = \biguplus_{e \in \mathsf E (\mathcal C)} \mathcal C_e ,
\end{equation}
and, if this is the case, then $\mc C$ is called a \,{\it Clifford semigroup}.

\smallskip
\begin{lemma} \label{5.3}
Let $G$ be a finite abelian group, $|G|>1$ and let $\Gamma\subset \Aut(G)$ be a subgroup. Suppose that $S_1,S_2\in \mathcal F(G)  \setminus \{1\}$, and for $i \in \{1,2\}$ let $[S_i] = [S_i]^{\mc F(G)}_{\mc B_\Gamma(G)}\in \mathcal C (\mathcal B_{\Gamma}(G),\mathcal F(G) )$.

\begin{enumerate}
\item
$[S_1]=[S_2]$ if and only if $\sigma_{\Gamma}(S_1)=\sigma_{\Gamma}(S_2)$.

\smallskip
		
\item
$[S_1]$ is an idempotent in $\mathcal C (\mathcal B_{\Gamma}(G),\mathcal F(G) )$ if and only if $\sigma_{\Gamma}(S_1)$ is a subgroup of \,$G$. If this is the case, then $S_1\in \mathcal B_{\Gamma}(G)$.
\end{enumerate}
\end{lemma}

\begin{proof}
1. Assume first that $\sigma_{\Gamma}(S_1)=\sigma_{\Gamma}(S_2)$, and let $S\in \mathcal F(G)$. By symmetry, it suffices to prove that $SS_1\in \mathcal B_{\Gamma}(G)$ implies $SS_2\in \mathcal B_{\Gamma}(G)$. If $SS_1\in \mathcal B_{\Gamma}(G)$, then
\[
0\in \sigma_{\Gamma}(SS_1)=\sigma_{\Gamma}(S)+\sigma_{\Gamma}(S_1)=\sigma_{\Gamma}(S)+\sigma_{\Gamma}(S_2)=\sigma_{\Gamma}(SS_2),
\]
and therefore $SS_2\in \mathcal B_{\Gamma}(G)$.

	Suppose that $\sigma_{\Gamma}(S_1)\neq\sigma_{\Gamma}(S_2)$. By symmetry, we may assume that  there exists $g\in G$ such that $g\in \sigma_{\Gamma}(S_1)\setminus\sigma_{\Gamma}(S_2)$. Therefore
	$S_1(-g)\in \mathcal B_{\Gamma}(G)$. If  $S_2(-g)\in \mathcal B_{\Gamma}(G)$, then there exists $\gamma\in \Gamma$ such that $\gamma(g)\in \sigma_{\Gamma}(S_2)$. Since $\Gamma$ is a group, we have $g=\gamma^{-1}(\gamma(g))\in \sigma_{\Gamma}(S_2)$, a contradiction. Thus $S_2(-g)\not\in \mathcal B_{\Gamma}(G)$, and therefore $[S_1]\neq [S_2]$.

\smallskip	

2. Let first $[S_1]$ be an idempotent in $\mathcal C(\mathcal B_{\Gamma}(G),\mathcal F(G))$. Then $[S_1] =2[S_1] =[S_1^2]$, and therefore  $\sigma_{\Gamma}(S_1)=\sigma_{\Gamma}(S_1^2)=\sigma_{\Gamma}(S_1)+\sigma_{\Gamma}(S_1)$. It follows that $\sigma_{\Gamma}(S_1)$ is a group, whence $0\in \sigma_{\Gamma}(S_1)$ and thus $S_1\in \mathcal B_{\Gamma}(G)$.	

Suppose now  that $\sigma_{\Gamma}(S_1)$ is a group. Then $\sigma_{\Gamma}(S_1)
=\sigma_{\Gamma}(S_1)+\sigma_{\Gamma}(S_1)=\sigma_{\Gamma}(S_1^2)$, and therefore $[S_1] = [S_1^2] =2[S_1]$ is an idempotent in $\mathcal C(\mathcal B_{\Gamma}(G),\mathcal F(G))$.
\end{proof}

\smallskip
\begin{theorem}\label{5.5}
Let $G = C_2^r\oplus C_4^t$, where $r\in \N_0$ and $t\in \N$. We set $\mc C = \mc C(\mathcal B_{\pm}(G),\mathcal F(G) )$, and for $S \in \mc F(G)$, we set $[S] = [S]^{\mc F(G)}_{\mc B_\pm(G)} \in\mathcal C$.

\smallskip

\begin{enumerate}
\item
Let $S \in \mc F(G)$. Then the following statements are equivalent.

\smallskip

\begin{enumerate}
\item
$[S] \in \mathsf E(\mc C)$.

\smallskip

\item
$S \in \mc B_\pm(G)$.

\smallskip

\item
$\sigma_\pm(S)$ is a subgroup of $2G$.
\end{enumerate}

\smallskip

\item
Let $S_1,\,S_2 \in \mc B_\pm(G)$. Then $[S_1] \le [S_2] \ ($in the Rees order of $\mathsf E(\mc C)\,)$ if and only if $\sigma_\pm(S_1) \supset \sigma_\pm(S_2)$. Moreover, $\mathcal C_{[S_1]} = \mathcal C_{[S_2]}$ if and only if $\sigma_\pm(S_1) = \sigma_\pm(S_2)$.

\smallskip

\item
$\mc C$ is a Clifford semigroup and $\{\mc C_{[S]} \colon S \in \mc B_\pm(G)\}$ is the set of its constituent groups.

\smallskip

\item
Let $G_1\subset 2G$ be a subgroup, let $g_1, \ldots, g_s \in G$ such that $(2g_1, \ldots, 2g_s)$ is a basis of $G_1$, and set $S_0 = g_1^2\cdot \ldots \cdot g_s^2 \in \mc B_\pm(G)$. Then
\[\qquad \quad
\sigma_\pm(S_0) = G_1 \quad \text{and} \quad\mc C_{[S_0]} = [S_0] +\bigl\{[g] \colon g \in G \ \text{ such that } \ 2g \in G_1\bigr\} \,\cong\, C_2^{r+t}.
\]
\end{enumerate}
\end{theorem}

\begin{proof}
1. The implications \,(a) $\Longrightarrow$ (b)\, and \,(c) $\Longrightarrow$ (a)\, follow from Lemma \ref{5.3}.2.

\smallskip

(b) $\Longrightarrow$ (c)\, Let $S=g_1\cdot\ldots\cdot g_{\ell}h_1\cdot\ldots\cdot h_{\ell'}\in \mathcal B_{\pm}(G)$, where $\ell,\ell'\in \N_0$, and $g_1,\ldots,g_{\ell},\,h_1, \ldots, h_{\ell'}\in G$ such that $\ord(g_i) \le 2$ for all $i \in [1,\ell]$ and $\ord(h_j)=4$ for all $j \in [1,\ell']$. By Lemma \ref{5.3}.2, $\sigma_{\pm}(S)$ is a subgroup of $G$. Since
	\[
0\in\sigma_{\pm}(S)=g_1+\ldots+g_{\ell}+\{h_1,-h_1\}+\ldots+\{h_{\ell'}, -h_{\ell'}\}=\sigma(S)+\sum_{i=1}^{\ell'}\{0, 2h_i\}\subset \sigma(S)+2G,
	\]
	we obtain $\sigma(S)\in 2G$, and hence $\sigma_{\pm}(S)\subset 2G$.

\smallskip

2. Since $[S_1],\,[S_2] \in \mathsf E(\mc C)$, we obtain, by the very definition of the Rees order, that
\[
\begin{aligned}
 \ [S_1]  \le [S_2]  & \iff [S_1S_2] = [S_1]+[S_2]=[S_1]  \iff  \sigma_\pm(S_1)+\sigma_\pm(S_2) = \sigma_\pm(S_1S_2) = \sigma_\pm(S_1) \\ & \iff \sigma_\pm(S_1) \supset \sigma_\pm(S_2) \,.
\end{aligned}
\]
If $S_1,\,S_2 \in \mc B_\pm(G)$, then $[S_1],\,[S_2] \in \mathsf E(\mc C)$. Hence $\mc C_{[S_1]} = \mc C_{[S_2]}$ is equivalent to $[S_1]=[S_2]$, and by Lemma \ref{5.3}.1 this holds if and only if $\sigma_\pm(S_1) = \sigma_\pm(S_2)$.

\smallskip

3. The monoid $\mathcal B_{\pm}(G)$ is seminormal by Theorem \ref{5.2}.2, and $\mathsf E(\mc C) = \{[S] \colon S \in \mc B_\pm(G) \}$ by 1. Hence, it follows from \cite[Theorem 1.1]{Ge-Zh19c} that $\mc C$ is a Clifford semigroup, and that $\{\mc C_{[S]} \colon S \in \mc B_\pm(G)\}$ is the set of its constituent groups.

\smallskip

4. Apparently, $\sigma_\pm(S_0) =\{0,2g_1\} + \ldots + \{0,2g_s\} = G_1$, whence $S_0 \in \mc B_\pm(G)$ and $[S_0]\in \mathsf E(\mc C)$. Let $g\in G$ such that $2g\in G_1$. Then $([g]+[S_0])+[S_0]=g+[S_0]$, and $([g]+[S_0])+[g] = [S_0]$, since  $\sigma_\pm(g^2S_0)=\{0, 2g\}+G_1=G_1 =\sigma_\pm(S_0)$. Therefore, we obtain $[g]+[S_0]\in \mathcal C_{[S_0]}$, and consequently
	\[
	\{[g]+[S_0]\colon g\in G \text{ such that }2g\in G_1\}\subset \mathcal C_{[S_0]}\,.
	\]
As to the opposite inclusion, let $S = h_1 \cdot \ldots \cdot h_\ell \in \mc F(G)$ (where $\ell \in \N$ and $h_1, \ldots, h_\ell \in G$) such that $[S] \in \mc
 C_{[S_0]}$. Since $S^2 \in \mc B_\pm(G)$ \,(hence $[S^2] \in \mathsf E(\mc C)$) and
$[S^2]= 2[S] \in \mc C_{[S_0]}$, it follows that $[S^2] = [S_0]$ and therefore $2\sigma(S) \in \sigma_\pm( S^2) = \{0,2h_1\}+ \ldots +\{0,2h_\ell\} = \sigma_{\pm}(S_0)=G_1$.

Now $[S] \in \mc C_{[S_0]}$ implies $[SS_0] = [S]+[S_0] = [S]$, and therefore
\[
\sigma_{\pm}(S)=\sigma_{\pm}(SS_0)=\sigma_{\pm}(S)+G_1=\sigma(S)+\sum_{i=1}^{\ell}\{0,2h_i\}+G_1=\sigma(S)+G_1=\sigma_{\pm}(\sigma(S)S_0)\,.
\]
It follows that $[S]=[\sigma(S)]+[S_0]$, whence $[S] \in \{[g]+[S_0]\colon g\in G \text{ such that }2g\in G_1\}$.

It remains to prove that $\mc C_{[S_0]} \cong C_2^{r+t}$. If $g \in G$ and $2g \in G_1$, then $2([g]+[S_0]) = [g^2S_0] = [S_0]$. Hence $\mc C_{[S_0]}$ is an
elementary  $2$-group, and it suffices to prove that $|\mc C_{[S_0]}|= 2^{r+t}$.
Since $G \cong C_2^r \oplus C_4^t$, we get $|\{g+G_1 \colon g\in G, \ 2g\in G_1\}|=2^{r+t}$. If $g_1,\,g_2 \in G$ are such that $2g_1,\,2g_2 \in G_1$, then
\[
g_1+G_1 =g_2+G_1 \,\iff\, \sigma_\pm(g_1S_0) = \pm g_1+G_1 = \pm g_2 +G_1= \sigma_\pm(g_2S_0) \,\iff\, [g_1]+[S_0] = ]g_2]+[S_0].
\]
Hence $|\mc C_{[S_0]}|= 2^{r+t}$.
\end{proof}

\medskip
\section{On the arithmetic  of  $\mathcal B_{\pm} (G)$} \label{6}
\medskip

In this section,  we study the arithmetic of $\mathcal B_{\pm} (G)$ and we post two problems (Problem \ref{6.7} and \ref{6.10}), continuing investigations started in \cite{B-M-O-S22}. So far, it has turned out that many of the structural results valid for $\mathcal B (G)$ also hold true for $\mathcal B_{\pm} (G)$, but there are also striking differences (see the discussion after Theorem \ref{6.2}).

Let $H$ be a BF-monoid. We start with two arithmetic invariants which are closely connected with the catenary degree. The set
\[
\Delta (H) = \bigcup_{L \in \mathcal L (H)} \Delta (L) \ \subset \N
\]
is the {\it set of distances} of $H$.  For $a \in H$, let $\omega (H,a)$ denote the smallest $N \in \N_0 \cup \{\infty\}$ with the following property:
\begin{itemize}
\item[] For all $n \in \N$ and $a_1, \ldots, a_n \in H$ with $a \t a_1 \cdot \ldots \cdot a_n$, there is a subset $\Omega \subset [1, n]$ such that $|\Omega| \le N$ and $a \t \prod_{\lambda \subset \Omega} a_{\lambda}$.
\end{itemize}
We set
\[
\omega (H) = \sup \{ \omega (H, u) \colon u \in \mathcal A (H) \} \in \N_0 \cup \{\infty\} \,.
\]
By definition, we have $\omega (H, u)=1$ if and only if $u \in H$ is a prime element, and $H$ is factorial but not a group if and only if $\omega (H)=1$. There is a canonical chain of inequalities (\cite[Proposition 3.6]{Ge-Ka10a})
\begin{equation} \label{inequality-1}
2 + \sup \Delta (H) \le \mathsf c (H) \le \omega (H) \,,
\end{equation}
and, in general, all inequalities can be strict, even for finitely generated monoids or for Dedekind domains.
To recall the concept of elasticities, suppose that $H \ne H^{\times}$ and let $k \in \N_0$. Then
\[
\mathcal U_k (H) = \bigcup_{k \in L, \ L \in \mathcal L (H)} L \ \subset \N
\]
denotes the {\it union of sets of lengths} containing $k$, $\lambda_k (H) = \min \mathcal U_k (H)$,  and
\[
\rho_k (H) = \sup \mathcal U_k (H) \quad \text{is  the {\it $k$-th elasticity}  of $H$}.
\]

We continue with a simple characterization of half-factoriality. It  formally coincides with the characterization of half-factoriality of $\mathcal B (G)$. By definition of half-factoriality $H$ is half-factorial if and only if $\Delta (H)=\emptyset$ if and only if $\rho_k (H)=k$ for all $k \in \N$.

\smallskip
\begin{lemma} \label{6.1}
Let $G$ be a finite abelian group. Then the following statements are equivalent$\,:$.
\begin{enumerate}
\item[(a)] $\mathsf c \big( \mathcal B_{\pm} (G) \big) \le 2$.

\smallskip

\item[(b)] $\Delta \big( \mathcal B_{\pm} (G) \big) = \emptyset$.

\smallskip

\item[(c)] $|G| \le 2$.
\end{enumerate}
\end{lemma}

\begin{proof}
(a) $\Longrightarrow$ (b) By \eqref{inequality-1}.

(b) $\Longrightarrow$ (c) Assume to the contrary that $|G|\ge 3$. If there is an element $g \in G$ with odd order $\ord (g)=n \ge 3$, then $g^n$ and $(-g)g$ are atoms and $\mathsf L_{\mathcal B_{\pm} (G)} \big( g^n (-g)^n \big) = \{2,n\}$. If there is an element $g \in G$ with order $\ord (g)=n \ge 4$, then $U = (-2g) g^2$ is an atom and $\mathsf L_{\mathcal B_{\pm} (G)} \big( (-U)U \big) = \{2,3\}$. If $G$ is an elementary $2$-group of rank $\mathsf r (G) \ge 2$, then there are two independent elements $e_1, e_2 \in G$, Then $U = e_1e_2(e_1+e_2)$ is an atom and $\mathsf L_{\mathcal B_{\pm} (G)} \big( (-U)U \big) = \{2,3\}$. Thus, in all cases,  we get $\Delta(\mc B_\pm(G))\ne \emptyset$.

(c) $\Longrightarrow$ (a) If $|G| \le 2$, then $\mathcal B_{\pm} (G) = \mathcal B (G)$ is factorial, whence $\mathsf c \big( \mathcal B_{\pm} (G) \big) \le 2$.
\end{proof}

From now on, we suppose that $\mathcal B_{\pm} (G)$ is not half-factorial and investigate the invariants $\mathcal U_k \big(\mathcal B_{\pm}(G) \big)$. Recall that $1 + \mathsf d (G) = \mathsf D (G)$.

\smallskip
\begin{theorem} \label{6.2}
Let $G$ be a  finite abelian group of odd order such that $\mathsf D(G)=\mathsf D^*(G)\ge 3$, and let $k=\ell \mathsf D(G)+j \ge 2$, where $\ell\in \N_0$ and $j\in [0,\mathsf d(G)]$. Then, we have
\[
\mathcal U_{k} \big(\mathcal B_{\pm}(G) \big) =
\begin{cases}
	\big[2, \lfloor k\mathsf D(G)/2\rfloor\big], &\text{ if } j\in [2,\mathsf d(G)] \text{ and }\ell=0;\\
	\big[2\ell, \lfloor k\mathsf D(G)/2\rfloor\big], &\text{ if } j=0  \text{ and }\ell\ge 1;\\
	\big[2\ell+1, \lfloor k\mathsf D(G)/2\rfloor\big], &\text{ if } j\in [1, \mathsf d(G)/2]  \text{ and }\ell\ge 1;\\
	\big[2\ell+2, \lfloor k\mathsf D(G)/2\rfloor\big], &\text{ if } j\in [1 + \mathsf d(G)/2, \mathsf d (G)]  \text{ and }\ell\ge 1.
\end{cases}
\]
\end{theorem}

\begin{proof}
Main parts of the claim are based on prior work done in \cite{B-M-O-S22}.

To begin with, the unions $\mathcal U_k \big(\mathcal B_{\pm}(G) \big)$ are intervals for all $k\in \N$ by \cite[Theorem 5.2]{B-M-O-S22}. Thus, it remains to determine their maxima $\rho_k \big(\mathcal B_{\pm}(G) \big)$ and their minima $\lambda_k \big(\mathcal B_{\pm}(G) \big)$.
We use that $\mathsf D \big(\mathcal B_{\pm}(G) \big)=\mathsf D(G)\ge 3$ because $|G|$ is odd (\cite[Corollary 6.2]{B-M-O-S22}).

\smallskip
\noindent
{\bf 1.} On $\rho_k \big(\mathcal B_{\pm}(G) \big)$. We need to prove that
\[
\rho_k \big(\mathcal B_{\pm}(G) \big)=\lfloor k\mathsf D(G)/2 \rfloor \qquad \text{ for all } \qquad k\ge 2 \,.
\]
A simple counting argument (\cite[Theorem 5.7]{B-M-O-S22}) shows that, for all $k \in \N$,
\[
\rho_{2k} \big(\mathcal B_{\pm}(G) \big) = k \mathsf D (G) \quad \text{and} \quad \rho_{2k+1} \big(\mathcal B_{\pm}(G) \big) \le  k \mathsf D (G) + \Big\lfloor \frac{\mathsf D (G)}{2} \Big\rfloor \,.
\]
Thus, it remains to show that, for $k \in \N$,
\[
\rho_{2k+1} \big(\mathcal B_{\pm}(G) \big) \ge  k \mathsf D (G) + \Big\lfloor \frac{\mathsf D (G)}{2} \Big\rfloor \,.
\]
Let $k\in \N$. Suppose that $G\cong C_{n_1}\oplus \ldots\oplus C_{n_r}$, where $r\in \N$ and $1<n_1\t \ldots\t n_r$, and let $(e_1,\ldots,e_r)$ be a basis of $G$ with $\ord(e_i)=n_i$ for all $i\in[1,r]$. Since $|G|$ is odd and $\mathsf D(G)=\mathsf D^*(G)$,
 \[
U_1 = e_1^{n_1-1}\cdot\ldots\cdot e_r^{n_r-1}(e_1+\ldots+e_r) \quad \text{and} \quad U_2 =(2e_1)^{n_1-1}\cdot\ldots\cdot (2e_r)^{n_r-1}(2e_1+\ldots+2e_r)
\]
are atoms of $\mathcal B(G)$ of length $\mathsf D(G)$. Again, because $|G|$ is odd, we have $\mathcal A \big( \mathcal B (G) \big) \subset \mathcal A \big( \mathcal B_{\pm} (G) \big)$ by  \cite[Theorem 6.1]{B-M-O-S22}, and therefore $U_1$ and $U_2$ are atoms of $\mathcal B_{\pm}(G)$ of length $\mathsf D(\mathcal B_{\pm}(G))$.
Since
\[
V_1 =(e_1+\ldots+e_r)^2 \quad \text{ and } \quad V_2 =V_1 (2e_1+\ldots+2e_r)
\]
are both atoms of $\mathcal B_{\pm}(G)$, we obtain the following equation, with products of atoms on both sides,
\begin{align*}
U_1^{k}  U_1^{k}  U_2
 = (e_1^2)^{k(n_1-1)}\cdot \ldots\cdot (e_r^2)^{k(n_r-1)} V_1^{k-1} V_2 \big( (2e_1)^2 \big)^{(n_1-1)/2}\cdot\ldots\cdot \big((2e_r)^2 \big)^{(n_r-1)/2}\,,
\end{align*}
whence
\[
\rho_{2k+1} \big(\mathcal B_{\pm}(G) \big)\ge \sum_{i=1}^{r}k (n_i-1)+k-1+1+\sum_{i=1}^r(n_i-1)/2 = k \mathsf D (G) + \lfloor \mathsf D(G)/2 \rfloor \,.
\]
\smallskip
\noindent
{\bf 2.} On $\lambda_k \big(\mathcal B_{\pm}(G) \big)$. Let $k=\ell \mathsf D(G)+j \ge 2$, where $\ell\in \N_0$ and $j\in [0,\mathsf d(G)]$. Since $\{ \id_G, - \id_G\}  \subset \Aut (G)$ is a subgroup, \cite[Theorem 5.8]{B-M-O-S22} implies that
\begin{equation} \label{lambda-formula}
\lambda_k \big(\mathcal B_{\pm}(G) \big) =
\begin{cases}
2 \ell \quad  & \text{for} \quad j = 0 \,, \\
2 \ell+1 \quad  & \text{for} \quad j \in [1, \rho_{2\ell+1} \big(\mathcal B_{\pm}(G) \big) - \ell \mathsf D (G)] \,, \\
2 \ell+2 \quad  & \text{for} \quad j \in [\rho_{2 \ell+1} \big(\mathcal B_{\pm}(G) \big) - \ell \mathsf D (G)+1, \mathsf d (G) ]  \,.
\end{cases}
\end{equation}

If $\ell=0$, then $j\in [2,\mathsf d(G)]=[\rho_{2\ell+1}(\mathcal B_{\pm}(G))-\ell \mathsf D(G)+1,\mathsf d(G)]$, and \eqref{lambda-formula} implies that $\lambda_k(\mathcal B_{\pm}(G))=2$. Suppose now that $\ell\ge 1$.

If $j=0$, then \eqref{lambda-formula} implies $\lambda_k(\mathcal B_{\pm}(G))=2\ell$.

If $j\in [1, \mathsf d(G)/2]=[1, \rho_{2\ell+1}(\mathcal B_{\pm}(G))-\ell \mathsf D(G)]$, then \eqref{lambda-formula} implies $\lambda_k(\mathcal B_{\pm}(G))=2\ell+1$.

If $j\in [ (\mathsf D(G)+1)/2,\mathsf d(G)]=[\rho_{2\ell+1}(\mathcal B_{\pm}(G))-\ell \mathsf D(G)+1,\mathsf d(G)]$, then \eqref{lambda-formula} implies that $\lambda_k(\mathcal B_{\pm}(G))=2\ell+2$.
\end{proof}

Thus, in particular,  all unions $\mathcal U_k \big( \mathcal B_{\pm} (G) \big)$ are intervals, as it is the case for the  unions $\mathcal U_k \big( \mathcal B (G) \big)$. Under very mild assumptions, unions of sets of lengths of BF-monoids are well-structured (see \cite{Tr19a} for a result in this direction). However, even for numerical monoids (which are finitely generated C-monoids as well as weakly Krull monoids), they need not be intervals. Furthermore, the maxima of the unions behave differently for the monoids $\mathcal B (G)$ and $\mathcal B_{\pm} (G)$. Indeed, for all cyclic groups $G$ of odd order and for all $k \in \N$, we have (see  Theorem \ref{6.2} and \cite[page 75, Theorem 5.3.1]{Ge-Ru09}),
\[
\rho_{2k+1} \big( \mathcal B (G) \big) = k|G| +1 \quad \text{and} \quad \rho_{2k+1} \big( \mathcal B_{\pm} (G) \big) = k|G| + \lfloor |G|/2 \rfloor \,.
\]
Theorem \ref{6.2} offers the precise values of $\rho_k \big( \mathcal B_{\pm} (G) \big)$ (in terms of the Davenport constant) for all $k \in \N$. The precise values of $\rho_{2k+1} \big( \mathcal B (G) \big)$, with $k \in \N$, are known only for very special groups (\cite{Ge-Gr-Yu15}).

\smallskip
\begin{lemma} \label{6.3}
Let $G$ be a cyclic group of odd order $n\ge 3$. Then we have $\{2, j \} \in \mathcal L \big( \mathcal B_{\pm}(G) \big)$ for every $j \in [3, n]$, and consequently $[1,n-2]\subset \Delta \big(\mathcal B_{\pm}(G) \big)$.
\end{lemma}

\begin{proof}
It suffices to show that, for every $k\in [1, n-1]$, there exists atoms $U,V\in \mathcal A \big( \mathcal B_{\pm}(G) \big)$ such that $\mathsf L_{\mathcal B_{\pm} (G)}(UV)=\{2,n-k+1\}$.
	
	Let $g\in G$ with $\ord(g)=n$ and let $k\in [1,n-1]$.
	Since $\{g^n, g^{n-k}(kg)\}\subset \mathcal A(G)$, it follows from  \cite[Theorem 6.1]{B-M-O-S22} that $\{ g^n, g^{n-k}(kg)\}\subset \mathcal A \big( \mathcal B_{\pm}(G) \big)$. Suppose $g^n\cdot g^{n-k}(kg)=V_1\cdot\ldots\cdot V_{\ell}$, where $\ell\in \N$ and $V_1,\ldots,\cdot V_{\ell}\in \mathcal A \big( \mathcal B_{\pm}(G) \big)$, and
	note that
\[
\big\{S\in \mathcal A \big( \mathcal B_{\pm}(\{g,kg\}) \big) \colon \mathsf v_{kg}(S)=1\}=\{(kg)g^{n-k}, (kg)g^k \big\} \,.
\]
If there exists some $i\in [1,\ell]$, say $i=1$, such that $V_1=(kg)g^{n-k}$, then $\ell=2$ and $V_2=g^n$. Otherwise there exists some $i\in [1,\ell]$, say again $i=1$, such that $V_1=(kg)g^k$. Then $V_2\cdot\ldots\cdot V_{\ell}=g^{2n-2k}=(g^2)^{n-k}$ and hence $\ell=n-k+1$.
	Therefore, we obtain $\mathsf L_{\mathcal B_{\pm} (G)} \big(g^n\cdot g^{n-k}(kg) \big)=\{2,n-k+1\}$.
\end{proof}

If $G$ is a finite abelian group with $|G| \ge 3$, then
\begin{equation} \label{inequality-2}
2 + \sup \Delta \big( \mathcal B (G) \big)  \le \mathsf c \big( \mathcal B (G) \big) \le \omega \big( \mathcal B (G) \big) = \mathsf D (G) \,.
\end{equation}
Since the inclusion $\mathcal B (G) \hookrightarrow \mathcal F (G)$ is a divisor theory, the last equality is straightforward and all the inequalities follow from \eqref{inequality-1}.
The first inequality (on the left side) is an equality if $\mathsf D (G) = \mathsf D^* (G)$ (see \eqref{daven-1} and \cite[Corollary 4.1]{Ge-Gr-Sc11a}). In our next result, we determine $\omega \big( \mathcal B_{\pm} (G) \big)$ for a prime cyclic group $G$ of order $|G|\ge 3$. Since $\mathcal B_{\pm} (G)$ is not Krull by Theorem \ref{4.4}, the inclusion $\mathcal B_{\pm} (G) \hookrightarrow \mathcal F (G)$ is not a divisor homomorphism, whence there exist sequences $A, B \in B_{\pm} (G)$ such that $A \t B$ in $\mathcal F (G)$ but $A \nmid B $ in $\mathcal B_{\pm} (G)$.

\smallskip
\begin{lemma} \label{special}
	Let $G$ be prime cyclic of order $|G|= p \ge 3$ and let $S$ be a sequence over $G^{\bullet}$ of length $|S|\ge p-1$. Then $\sigma_{\pm}(S)=G$. In particular, $S\in \mathcal B_{\pm}(G)$.
\end{lemma}
\begin{proof}
	Suppose $S=g_1\cdot\ldots\cdot g_{\ell}$, where $\ell=|S|\ge p-1$ and $g_1,\ldots,g_{\ell}\in G^{\bullet}$.
It follows from the  Cauchy-Davenport Theorem (\cite[Theorem 6.2]{Gr13a}) that
	\[
	|\sigma_{\pm}(S)|=|\{g_1,-g_1\}+\ldots+\{g_{\ell},-g_{\ell}\}|\ge \min\{p, 2\ell-(\ell-1)\}=p\,,
	\]
	whence $\sigma_{\pm}(S)=G$ and hence $S\in \mathcal B_{\pm}(G)$.
\end{proof}

\smallskip
\begin{theorem} \label{6.4}
Let $G$ be prime cyclic of order $|G|= p \ge 3$. Then $\Delta \big( \mathcal B_{\pm} (G)  \big) = [1,p-2]$ and
\[
2 + \max \Delta \big( \mathcal B_{\pm} (G)  \big)  = \mathsf c \big( \mathcal B_{\pm} (G) \big) = \omega \big( \mathcal B_{\pm} (G) \big) = \mathsf D_{\pm} (G) \,.
\]
\end{theorem}

\begin{proof}
It is well-known that $\mathsf D (G)=p$ and, by \cite[Corollary 6.2]{B-M-O-S22}, we have $\mathsf D (G) = \mathsf D_{\pm} (G)$. If $g \in G^{\bullet}$, then
\[
\mathsf L_{\mathcal B_{\pm} (G)} \big( (-g)^p g^p \big) = \{2, p \} \,,
\]
whence \eqref{inequality-1} implies that $p \le 2 + \max \Delta \big( \mathcal B_{\pm} (G)  \big) \le \mathsf c \big( \mathcal B_{\pm} (G) \big) \le  \omega \big( \mathcal B_{\pm} (G) \big)$. Since $[1,p-2] \subset \Delta \big( \mathcal B_{\pm} (G)  \big)$ by Lemma \ref{6.3}, it remains to show the following assertion.

\begin{enumerate}
\item[{\bf A.}]   $\omega \big( \mathcal B_{\pm} (G) \big) \le p$.
\end{enumerate}

{\it Proof of {\bf A.}} Let $U, V_1,\ldots, V_{\ell}\in \mathcal A \big(\mathcal B_{\pm}(G) \big)$ such that $U$ divides $V_1\cdot\ldots\cdot V_{\ell}$ in $\mathcal B_{\pm}(G)$. We must show that there exists some $I\subset [1,\ell]$ such that  $|I|\le p$ and $U$ divides $\prod_{i\in I}V_i$ in $\mathcal B_{\pm}(G)$.
	 If $\ell\le p$ or $U=0$, then the assertion is clear.
	
	  Suppose that $\ell>p$ and that $|U|\ge 2$. Without loss of generality, we may assume that $|V_i|\ge 2$ for all $i\in [1,\ell]$.
	  After renumbering if necessary, we may assume that there exists some $r\in [1,\ell]$ such that  $V_i=V_i'V_i''$ for each $i\in [1,r]$ and
	$U=V_1'\cdot\ldots\cdot V_r'$, where $V_i',V_i''\in \mathcal F(G)$  for all $i\in [1,r]$ and $V_1',\ldots, V_r'$ are nontrivial.
	Since $\mathsf D_{\pm}(G)=p$, we have $|U|\le p$, whence $r\le p$ and $U$ divides $V_1\cdot\ldots\cdot V_p$ in $\mathcal F(G)$.
	It suffices to prove that $U^{-1}V_1\cdot\ldots\cdot V_p \in \mc B_\pm(G)$. In fact, since $V_i$ does not divide $U$ in $\mathcal F(G)$ for every $i\in [1,p]$, we have $|U^{-1}V_1\cdot\ldots\cdot V_p|\ge p$ and hence the assertion follows from Lemma \ref{special}.
\end{proof}

\smallskip
Thus, in particular, $\Delta \big( \mathcal B_{\pm} (G) \big)$ is a finite interval if $G$ is prime cyclic. It is easy to verify that $\Delta (H)$ is finite with $\min \Delta (H) = \gcd \Delta (H)$ for all finitely generated monoids $H$. If $|G| \ge 3$, then $\Delta \big( \mathcal B (G) \big)$ is a finite interval with $\min \Delta \big( \mathcal B (G) \big) = 1$ (\cite{Ge-Yu12b}), and the same is true for   seminormal weakly Krull Mori monoids satisfying some further algebraic conditions (\cite{Ge-Zh16c}). On the other hand, sets of distances need not be intervals for numerical monoids and, if we do not impose any restrictions on the distribution of height-one prime ideals in the classes, then, for any given finite set $\Delta$ with $\min \Delta = \gcd \Delta$, there is a Dedekind domain $D$ with $\Delta (D) = \Delta$.  Since $\mathcal B_{\pm} (G)$, where $G$ is a prime cyclic group, is neither weakly Krull nor seminormal (see Theorems \ref{4.4} and \ref{5.2}), the statement that $\Delta \big( \mathcal B_{\pm} (G)  \big)$ is an interval does not follow from any earlier results in the literature.

\smallskip
\begin{corollary} \label{6.5}
The statements of {\rm Theorems \ref{6.2} and  \ref{6.4}} hold true  for the monoids $N_m^{(+)}(\mc O) \ ($discussed in {\rm Theorem \ref{3.2})} and for $R_m^{\prime \circ}(\Delta) \ ($discussed in \,{\rm Corollary \ref{3.6})}.
\end{corollary}

\begin{proof}
This follows from Proposition \ref{2.2}, Theorem \ref{3.2}, and Corollary \ref{3.6}.
\end{proof}

\smallskip
\begin{remark} \label{6.6}~

1. Let $G$ be a finite abelian group with $|G| \ge 3$. By \eqref{inequality-2}, we have $\mathsf c \big( \mathcal B (G) \big) \le \mathsf D (G)$. The extremal cases, when
\[
\mathsf c \big( \mathcal B (G) \big) =  \mathsf D (G) \quad \text{and} \quad \mathsf c \big( \mathcal B (G) \big) =  \mathsf D (G)-1
\]
are characterized in \cite{Ge-Zh15b}. But, apart from these extremal cases and a couple of small groups, the precise value of $\mathsf c \big( \mathcal B (G) \big)$ is not known yet.

2. Let $G$ be a finite group (not necessarily abelian) and let $\mathcal B (G)$ denote the monoid of product-one sequences over $G$. If $G$ is a dihedral group of order $|G|=2n$, where $n \ge 3$ is odd, then, by \cite[Theorems 4.1 and 5.1]{G-G-O-Z22a}, we have
\[
2 + \max \Delta  \big( \mathcal B (G) \big) = \mathsf c \big( \mathcal B (G) \big) = \omega \big( \mathcal B (G) \big) = \mathsf D \big( \mathcal B (G) \big) = 2n \,.
\]
The most difficult step, to get these equalities, is to prove that $\omega \big( \mathcal B (G) \big) \le  \mathsf D \big( \mathcal B (G) \big)$.

3. Let $G$ be a nontrivial group of odd order. Then $\mathcal A (\mathcal B (G)) \subset \mathcal A ( \mathcal B_{\pm} (G))$ and $\mathsf D (G) = \mathsf D_{\pm} (G)$ (\cite[Theorem 6.1 and Corollary 6.2]{B-M-O-S22}). If $U = g_1 \cdot \ldots \cdot g_{\ell} \in \mathcal A (\mathcal B (G)) $ with $|U|=\ell = \mathsf D (G)$, then $U$ divides $\big( (-g_1)g_1 \big) \cdot \ldots \cdot \big( (-g_{\ell})g_{\ell} \big)$ but no proper subproduct, whence $\omega \big( \mathcal B_{\pm} (G) \big) \ge \mathsf D_{\pm}  (G) $.
\end{remark}

\smallskip
\begin{problem} \label{6.7}~

\begin{enumerate}
\item Determine the finite abelian groups $G$ satisfying $\omega \big( \mathcal B_{\pm} (G) \big) \le \mathsf D_{\pm}  (G) $. If this inequality holds and $|G|$ is odd, then, by Remark \ref{6.6}.3 and \eqref{inequality-1}, we obtain that
      $\mathsf c \big( \mathcal B_{\pm} (G) \big) \le \omega \big( \mathcal B_{\pm} (G) \big) = \mathsf D_{\pm}  (G) $.

\item Is the set of distances $\Delta \big( \mathcal B_{\pm} (G) \big)$ an interval for all finite abelian groups?

\item Since $    \mathcal B_{\pm} (G) $ is finitely generated, its monotone catenary degree is finite (\cite{Fo06a}). Determine the monotone catenary degree $\mathsf c_{\mon} \big( \mathcal B_{\pm} (G) \big)$ and the groups $G$ for which $\mathsf c_{\mon} \big( \mathcal B_{\pm} (G) \big) = \mathsf c_{\mon} \big( \mathcal B (G) \big)$.
\end{enumerate}
\end{problem}

\smallskip
Let $H$ be a Krull monoid with finite class group $G$ and suppose that each class contains a prime divisor. Then there is a transfer homomorphism $\theta \colon H \to \mathcal B (G)$, whence $\mathcal L (H) = \mathcal L \big( \mathcal B (G) \big)$ and all sets of lengths $L \in \mathcal L (H)$ are almost arithmetic multiprogressions with global bounds on all parameters (\cite[Chapter 4]{Ge-HK06a}). In particular, sets of lengths in $H$ only depend on $G$. The associated inverse problem (known as the Characterization Problem) asks whether the system $\mathcal L \big( \mathcal B (G) \big)$ is characteristic for the group. Here is the precise formulation of this problem.

\smallskip
    \noindent
    {\bf Characterization Problem.} Let $G$ and $G'$ be finite abelian groups with $\mathsf D (G) \ge 4$ and $\mathcal L \big( \mathcal B (G) \big) = \mathcal L \big( \mathcal B (G') \big)$. Are $G$ and $G'$ isomorphic?

\smallskip
Affirmative answers were given for groups $G$ of rank $\mathsf r (G) \le 2$ and others.
The standing conjecture expects an affirmative answer in general (see \cite{Ge-Zh20a} for an overview of the state of the art). The analogue question was studied also for monoids of product-one sequences over finite, but not necessarily abelian groups (\cite[Corollary 6.13]{G-G-O-Z22a}). It would be interesting to understand whether similar results can be established for $\mathcal B_{\pm} (G)$ (see Problem \ref{6.10}). As a first step, we
 generalize a result valid for $\mathsf D (G)$ and $\mathcal B (G)$ to $\mathsf D_{\pm} (G)$ and $\mathcal B_{\pm} (G)$.

\smallskip
\begin{proposition} \label{6.8}
Let $G$ be a finite abelian group.
\begin{enumerate}
\item There are (up to isomorphism) only finitely many finite abelian groups $G'$ with $\mathsf D_{\pm} (G) = \mathsf D_{\pm} (G')$.

\item There are (up to isomorphism) only finitely many finite abelian groups $G'$ with $\mathcal L \big( \mathcal B_{\pm} (G) \big) = \mathcal L \big( \mathcal B_{\pm} (G') \big)$.
\end{enumerate}
\end{proposition}

\begin{proof}
1. If $G'$ is a finite abelian group, say $G' \cong C_{n_1} \oplus \ldots \oplus C_{n_r}$ with $r \in \N$ and $1 < n_1 \t \ldots \t n_r$, then \eqref{daven-2} shows that
\[
1 + \sum_{i=1}^t (n_i-1) + \sum_{i=t+1}^r \frac{n_i}{2} \le \mathsf D_{\pm} (G') \,,
\]
where $t \in [0,r]$ is maximal such that $2 \nmid n_i$. Thus, there are only finitely many finite abelian groups $G'$ (up to isomorphism) with $\mathsf D_{\pm} (G') = \mathsf D_{\pm} (G)$.

2. Let $G'$ be a finite abelian group with $\mathcal L \big( \mathcal B_{\pm} (G) \big) = \mathcal L \big( \mathcal B_{\pm} (G') \big)$. Then $\rho_2 \big( \mathcal B_{\pm} (G) \big) = \rho_2 \big( \mathcal B_{\pm} (G') \big)$, and \cite[Theorem 5.7]{B-M-O-S22} implies that
\[
\mathsf D_{\pm} (G) = \rho_2 \big( \mathcal B_{\pm} (G) \big) = \rho_2 \big( \mathcal B_{\pm} (G') \big)  = \mathsf D_{\pm} (G') \,.
\]
Thus, the assertion follows from 1.
\end{proof}

\smallskip
\begin{theorem} \label{6.9}
Let $G$ be a cyclic group of odd order $|G| \ge 3$.  If  $G'$ is any finite abelian group of odd order with $\mathcal L \big( \mathcal B_{\pm}(G') \big)=\mathcal L \big(\mathcal B_{\pm}(G) \big)$, then  $G$ and $G'$ are isomorphic.
\end{theorem}

\begin{proof}
We set $|G|=n$. Since $G$ and $G'$  have odd order, it follows from  \cite[Corollary 6.2]{B-M-O-S22} that $\mathsf D_{\pm}(G)=\mathsf D (G)=n$ and $\mathsf D_{\pm}(G')=\mathsf D(G')$.
Since $\mathcal L \big(\mathcal B_{\pm}(G') \big)=\mathcal L \big(\mathcal B_{\pm}(G) \big)$, we infer that
\[
n = \mathsf D_{\pm}(G)= \rho_2 \big(\mathcal B_{\pm}(G) \big)=\rho_2 \big(\mathcal B_{\pm}(G') \big) = \mathsf D_{\pm}(G')=\mathsf D(G') \,.
\]
By Lemma \ref{6.3}, we have $\{2,n\} \in  \mathcal L \big(\mathcal B_{\pm}(G) \big)=\mathcal L \big(\mathcal B_{\pm}(G') \big)$. Thus,  there exists $B=g_1\cdot\ldots\cdot g_{\ell}\in \mathcal B_{\pm}(G')$ such that
$\mathsf L_{ \mathcal B_{\pm} (G')}(B)=\{2,n\}$, where $\ell\in \N$ and $g_1,\ldots,g_{\ell}\in G'$. Therefore, there exist disjoint nonempty $I_1,I_2\subset [0,\ell]$ with $I_1\uplus I_2=[0,\ell]$ and disjoint nonempty $J_1,\ldots,J_n\subset [0,\ell]$ with $J_1\uplus\ldots \uplus J_n=[0,\ell]$ such that $\prod_{i\in I_1}g_i, \prod_{i\in I_2}g_i, \prod_{j\in J_1}g_j, \ldots, \prod_{j\in J_n}g_j\in \mathcal A \big(\mathcal B_{\pm}(G') \big)$.
Moreover there exist $\tau_1,\ldots,\tau_{\ell}\in \{\id_{G'}, -\id_{G'}\}$ such that
\[
A_1 =\prod_{i\in I_1}\tau_{i}(g_i) \qquad \text{ and  } \qquad A_2 =\prod_{i\in I_2}\tau_i(g_i)
\]
are zero-sum sequences. We set $B^*=\tau_1(g_1)\cdot\ldots\cdot \tau_{\ell}(g_{\ell})$ and  $U_i=\prod_{j\in J_i}\tau_j(g_j)$ for every $i\in[1,n]$. Then $B^*=A_1A_2=U_1\cdot\ldots\cdot U_n$. We proceed by the following claim.

\begin{enumerate}
\item[{\bf A.}] Let $V=h_1\cdot\ldots\cdot h_s\in \mathcal B_{\pm}(G')$ and let $\alpha_1,\ldots,\alpha_s\in \{\id_{G'}, -\id_{G'}\}$, where $s\in \N$ and $h_1,\ldots, h_s\in G'$. Then $V\in \mathcal A \big(\mathcal B_{\pm}(G') \big)$ if and only if $\alpha_1(h_1)\cdot\ldots\cdot \alpha_{s}(h_s)\in \mathcal A \big(\mathcal B_{\pm}(G') \big)$.
\end{enumerate}

\smallskip
{\it Proof of {\bf A.}}
Let $V_1=\alpha_1(h_1)\cdot\ldots\cdot \alpha_{s}(h_s)$. By taking the weights $\alpha_1,\ldots,\alpha_s$, we have $V_1\in \mathcal B_{\pm}(G')$.
Since $\{\id_{G'}, -\id_{G'}\} \subset \Aut (G')$ is a group, we only need to show that $V\in \mathcal A \big(\mathcal B_{\pm}(G') \big)$ implies that $V_1\in \mathcal A \big(\mathcal B_{\pm}(G') \big)$.

Suppose $V\in \mathcal A \big(\mathcal B_{\pm}(G') \big)$. Assume to the contrary that $V_1\not\in \mathcal A \big(\mathcal B_{\pm}(G') \big)$. Then there exist disjoint nonempty $I_1, I_2\subset [0,s]$ with $I_1\uplus I_2=[0,s]$ such that $\prod_{i\in I_1}\alpha_i(g_i), \prod_{i\in I_2}\alpha_i(g_i)\in \mathcal B_{\pm}(G')$, whence $\prod_{i\in I_1}g_i, \prod_{i\in I_2}g_i\in \mathcal B_{\pm}(G')$, a contradiction to $V\in \mathcal A \big(\mathcal B_{\pm}(G') \big)$.
  \qed ({\bf A})

\medskip
It follows from {\bf A} that $A_1,A_2, U_1,\ldots, U_{n}\in \mathcal A \big(\mathcal B_{\pm}(G') \big)$ and that, moreover, $\mathsf L_{\mathcal B_{\pm}(G')}(B^*)=\mathsf L_{\mathcal B_{\pm}(G')}(B)=\{2,n\}$.
Since $A_1,A_2$ are zero-sum sequences, we obtain $A_1,A_2\in \mathcal A \big(\mathcal B(G') \big)$.
A simple counting argument shows  that $|A_1|=|A_2|=n=\mathsf D(G')$.

If $|\supp(A_1)|=1$ or $|\supp(A_2)|=1$, then $G'\cong G$. Otherwise, there exist, after renumbering if necessary, elements $x\in \supp(A_1)$ and $y\in \supp(A_2)$ such that $x\neq y$.
Since $|x^{-1}A_1y|=\mathsf D(G')$ and $x^{-1}A_1y$ is not zero-sum, there exists an atom $A' \in \mathcal A \big( \mathcal B (G) \big)$ with $A'\t x^{-1}A_1y$ such that $|A'|<n$ and hence $|\mathsf L_{\mathcal B (G')} (A_1A_2)|\ge 2$.
By \cite[Theorem 6.1]{B-M-O-S22}, we have $\mathcal A \big( \mathcal B(G') \big) \subset \mathcal A \big( \mathcal B_{\pm}(G') \big)$, whence $\mathsf L_{\mathcal B (G')} (A_1A_2)\subset \mathsf L_{ \mathcal B_{\pm} (G')} (A_1A_2)=\{2,n\}$.
It follows that $\{2,n\}=\mathsf L_{\mathcal B (G')} (A_1A_2) \in \mathcal L \big( \mathcal B (G') \big)$. Therefore, \cite[Theorem 6.6.3]{Ge-HK06a} implies that  $G'$ is cyclic, whence $G$ and $G'$ are isomorphic.
\end{proof}

\smallskip
\begin{problem} \label{6.10}~
Which finite abelian groups $G$ have the following property:
\begin{itemize}
\item[] If $G'$ is a finite abelian group with $\mathcal L \big( \mathcal B_{\pm} (G) \big) = \mathcal L \big( \mathcal B_{\pm} (G') \big)$, then $G$ and $G'$ are isomorphic.
\end{itemize}
\end{problem}
The set of minimal distances $\Delta^* \big( \mathcal B (G) \big)$ is a key tool in all work on the Characterization Problem for $\mathcal B (G)$. Thus, a systematic study of $\Delta^* \big( \mathcal B_{\pm} (G) \big)$ should be a promising step to make progress on Problem \ref{6.10}.

\bigskip
\noindent
{\bf Acknowledgement.} The referee read the whole paper line by line and we want to thank for all his/her comments.

\providecommand{\bysame}{\leavevmode\hbox to3em{\hrulefill}\thinspace}
\providecommand{\MR}{\relax\ifhmode\unskip\space\fi MR }
\providecommand{\MRhref}[2]{%
  \href{http://www.ams.org/mathscinet-getitem?mr=#1}{#2}
}
\providecommand{\href}[2]{#2}

\end{document}